\documentclass[a4paper]{amsart}

\newcommand{\version}{2.0}
\date{10 Sept 2010}

\author[A. Fornasiero]{Antongiulio Fornasiero} 
\address{Institut f\"ur Mathematische Logik\\
    Einsteinstr.~62, 48149 M\"unster, Germany}
\email{antongiulio.fornasiero@googlemail.com}
\title[Lovely pairs, v.~\version]{Lovely pairs for independence relations\\
{\mdseries\footnotesize Version \version}}

\usepackage{ifpdf}

\usepackage{amsmath, amsthm, amssymb}
\usepackage{mathtools}
\usepackage{xspace}

\usepackage[neveradjust]{paralist}
\AtBeginDocument{\setdefaultleftmargin{0pt}{}{}{}{}{}}
\setdefaultenum{\upshape(1)}{}{}{}

\usepackage[T1]{fontenc}

\usepackage{comment}

\ifpdf
\usepackage[pdftex, 
pdfborder={0 0 0}, plainpages=false,%
pdfpagelabels, pdfdisplaydoctitle=true,%
pdfstartview={XYZ null null null},%
colorlinks=false, bookmarksnumbered=true%
]{hyperref}
\hypersetup{pdfauthor={A. Fornasiero},
  pdftitle={Lovely pairs v. \version},
  pdfkeywords={Lovely pairs, dense pairs,
  pregeometries, existential matroids},
  pdfsubject={Lovely pairs for independence relations},
  pdfduplex=DuplexFlipLongEdge}
\else
\usepackage[plainpages=false, pdfpagelabels, linktocpage=true, colorlinks=false]{hyperref}
\fi

\usepackage[alphabetic, backrefs, msc-links]{amsrefs}

\makeatletter
\@namedef{subjclassname@2010}{%
 \textup{2010} Mathematics Subject Classification}
\makeatother

\newcommand{\rom}{\textup}
\def\hyph{\nobreakdash-\hspace{0pt}\relax}

\newcommand*{\intro}[1]{\textbf{#1}}
\newcommand*{\Pa}[1]{\bigl( #1 \bigr)}
\newcommand*{\set}[1]{\{#1\}}
\newcommand*{\pair}[1]{\langle #1 \rangle}
\newcommand*{\abs}[1]{\lvert#1\rvert}
\newcommand*{\card}[1]{\lvert#1\rvert}
\newcommand*{\length}[1]{\lvert#1\rvert}
\newcommand{\Nat}{\mathbb{N}}

\newcommand{\Q}{\mathbb{Q}}
\newcommand{\Z}{\mathbb{Z}}

\newcommand{\rest}{\upharpoonright}
\DeclareMathOperator{\dcl}{dcl}
\DeclareMathOperator{\acl}{acl}
\DeclareMathOperator{\dcleq}{dcl^{\eq}}
\DeclareMathOperator{\acleq}{acl^{\eq}}
\newcommand{\Sdiff}{\mathbin \Delta}
\newcommand{\et}{\ \&\ }
\newcommand{\ett}{\quad \& \quad }

\newcommand{\Ra}{\quad \Rightarrow \quad}
\newcommand{\av}{\bar a}
\newcommand{\bv}{\bar b}
\newcommand{\cv}{\bar c}
\newcommand{\dv}{\bar d}
\newcommand{\fv}{\bar f}
\newcommand{\pv}{\bar p}
\newcommand{\x}{\bar x}
\newcommand{\y}{\bar y}
\newcommand{\z}{\bar z}
\newcommand{\Ffam}{\mathcal F}
\DeclareMathOperator{\Dom}{Dom}
\DeclareMathOperator{\Imag}{Im}
\DeclareMathOperator{\tp}{tp}
\DeclareMathOperator{\Theory}{Th}
\newcommand{\Ttwo}{T^2}
\newcommand{\Td}{T^d}
\newcommand{\Ltwo}{\Lang^2}
\newcommand{\Lang}{\mathcal L}
\newcommand{\monster}{\mathfrak C}
\newcommand{\Mb}{\monster}
\newcommand{\Mp}{\monster_P}
\newcommand{\M}{M}
\newcommand{\N}{N}
\DeclareMathOperator{\aut}{Aut}

\newcommand{\eq}{\mathrm{eq}}
\newcommand{\Meq}{M^\eq}
\newcommand{\Mpeq}{{M'}^\eq}
\newcommand{\monstereq}{\monster^\eq}
\newcommand*{\param}[1]{\ulcorner #1 \urcorner}
\DeclareMathOperator{\smallo}{o}
\DeclareMathOperator{\bigo}{O}
\DeclareMathOperator{\U}{U}
\newcommand{\Uind}{\U^{\ind}}

\DeclareMathOperator{\SU}{SU}
\newcommand{\Uacl}{\U^{\acl}}
\newcommand{\Up}{\U^{\mathrm P}}
\newcommand{\Utheta}{\U^\theta}
\DeclareMathOperator{\mat}{cl}
\newcommand{\mateq}{\mat^{\eq}}
\newcommand{\matP}{\mat^{\mathrm P}}
\newcommand{\mattheta}{\mat^{\theta}}
\def\Ind#1#2{#1\setbox0=\hbox{$#1x$}\kern\wd0\hbox to
  0pt{\hss$#1\mid$\hss}\lower.9\ht0\hbox to 0pt{\hss$#1\smile$\hss}\kern\wd0}
\newcommand*{\ind}[1][]{\mathop{\mathpalette\Ind{}^{\!\!\!\!\rlap{$\scriptscriptstyle\textnormal{#1}$}\,\,\,\,}}}

\def\indtheta{\ind[$\theta$]}
\def\indp{\ind_P}
\def\indf{\ind[f]}
\def\indacl{\ind[$\acl$]}
\def\indmat{\ind[$\mat$]}
\def\indeq{\ind[$\eq$]}
\newcommand*{\notind}[1][]{\mathrel{\not\mkern-7mu{\ind[#1]}}}

\newcommand{\forkext}{\mathrel{\sqsubseteq_{\mkern-15mu{\not}\mkern15mu}}}
\newcommand{\nforkext}{\sqsubseteq}
\newcommand{\elem}{\equiv}

\newcommand{\Ptp}{\ensuremath{\text{\rm P-$\tp$}}}
\newcommand{\Pindependent}{P\hyph independent\xspace}
\newcommand{\Ptype}{P\hyph type\xspace}
\newcommand{\zapplication}{Z-application\xspace}

\newcommand{\Wlog}{W.l.o.g\mbox{.}\xspace}
\newcommand{\wloG}{w.l.o.g\mbox{.}\xspace}

\newcommand{\ie}{i.e\mbox{.}\xspace}

\newcommand{\cf}{cf\mbox{.}\xspace}

\newcommand{\tfae}{t.f.a.e\mbox{.}\xspace}

\newcommand*{\Case}[1]{\par\medskip\noindent\textit{Case} #1:}

\newtheorem{lemma}{Lemma}[section]
\newtheorem{theorem}[lemma]{Theorem}
\newtheorem{corollary}[lemma]{Corollary}
\newtheorem{conjecture}[lemma]{Conjecture}
\newtheorem{proposition}[lemma]{Proposition}

\newtheorem{fact}[lemma]{Fact}
\newtheorem*{hypothesis*}{Hypothesis}
\newtheorem*{proviso*}{Proviso}
\newtheorem*{fact*}{Fact}

\theoremstyle{remark}

\newtheorem{claim}{Claim}
\newtheorem*{claim*}{Claim}

\theoremstyle{definition}
\newtheorem{definition}[lemma]{Definition}

\newtheorem{remark}[lemma]{Remark}
\newtheorem{example}[lemma]{Example}

\newtheorem{question}[lemma]{Question}

\newtheorem*{warning*}{Warning}

\newenvironment{sentence}[1][]{%
  \begin{list}{}{%
    \setlength\topsep{0.5ex}%
    \setlength\leftmargin{\parindent}%
  }%
  \item[#1]
 }
 {\end{list}}

\numberwithin{equation}{section}

\begin{document}

\begin{abstract}
In the literature there are two different notions of lovely pairs of a theory T, according to whether T is simple or geometric.
We introduce a notion of lovely pairs for an independence relation, which generalizes both the simple and the geometric case, and show how the main theorems for those two cases extend to our general notion.
\end{abstract}


\maketitle

{\small
\tableofcontents
}

\makeatletter
\renewcommand\@makefnmark%
   {\normalfont(\@textsuperscript{\normalfont\@thefnmark})}
\renewcommand\@makefntext[1]%
   {\noindent\makebox[1.8em][r]{\@makefnmark\ }#1}
\makeatother


\section{Introduction}

Let $T$ be a complete first-order theory and $\monster$ be a monster model for~$T$.
In the literature there are at least different notions of lovely pairs, according to
whether $T$ is simple \cite{poizat83, BPV, vassiliev05} or geometric
\cite{macintyre, vdd-dense, boxall}.
Another class of lovely pairs, generalizing the geometric case, is given by dense pairs of theories with an existential matroid (see \cite{fornasiero-matroids} for the case when $T$ expands a field).

The study of lovely pairs started with~\cite{robinson59}, where dense pairs of real closed fields and pairs of algebraically closed fields were studied, and has continued, under various names, until present day: \cite{macintyre} studied lovely pairs of geometric theories and \cite{vdd-dense} lovely pairs of o-minimal structure expanding a group (they called them ``dense pairs'', because when $T$ is an o-minimal theory expanding a group, a lovely pair for $T$ is a pair $A \prec B \models T$, with $A$ dense subset of~$B$);
on the other hand, \cite{poizat83} studied beautiful pairs of stable structures, which were generalized to lovely pairs of simple structures in~\cite{BPV}.

Many results and techniques are similar for all the above classes of theories.
In this article we introduce a unified approach, via a general notion, the $\ind$-lovely pairs (Definition~\ref{def:lovely}), where $\ind$ is an independence relation on~$\monster$ in the sense of~\cite{adler}, and show that for suitable values of $\ind$ we get the various special cases we recorded above:
\begin{enumerate}
\item
if $\monster$ is simple and $\indf$ is Shelah's forking, then a $\indf$-lovely pair is a lovely pair of a simple theory in the sense of~\cite{BPV};
\item 
if $\monster$ is geometric and $\indacl\,$ is the independence relation induced by~$\acl$, then a $\indacl$-lovely pair is a lovely pair of a geometric theory in the sense of \cite{boxall};
\item 
if $\monster$ has an existential matroid $\mat$ and expands a field and $\indmat$ is the independence relation induced by~$\mat$, then a $\indmat$-lovely pair is a dense pair in the sense of~\cite{fornasiero-matroids}.
\end{enumerate}

It may happen that $\monster$ has more than one independence relation: in this case, to different independence relations correspond different notions of lovely pairs. 
For instance, $\monster$~can be both stable and geometric, but $\indf \neq \indmat$: in this case, a lovely pair of $T$ as a geometric theory will be different from a lovely pair of $T$ as a simple theory: see Example~\ref{ex:triple-love}.

We generalize some of the main results for lovely pairs from \cite{vdd-dense, BPV, boxall} to $\ind$-lovely pairs:
see \S\ref{subsec:existence}, \ref{subsec:uniqueness}, \ref{sec:model-complete}, \ref{sec:tuples}.
Moreover, we show how lovely pairs inherit ``stability'' properties from~$T$:
that is, if $T$ is stable, or simple, or NIP, then lovely pairs of $T$ have the same property, see \S\ref{sec:stability}.

Since $\ind$-lovely pairs depend in an essential way on the independence relation~$\ind$, we need a more detailed study of independence relations, which we do in \S\ref{sec:preliminary}.
Moreover, different independence relations give different kinds of lovely pairs: thus, it seems worthwhile to produce new independence relations; a technique to produce new independence relations is explained in \S\ref{sec:new-independence}.
Finally, a $\ind$-lovely pair has at least one independence relation $\ind_P$ inherited 
from~$\monster$ (see \S\ref{sec:lovely-independence}): we hope that, among
other things, $\ind_P$ will prove useful in studying the original theory~$T$.

A notion of $\ind$-lovely pairs has been also proposed by I. Ben Yaacov, using
a different notion of ``independence relation'' than the one employed here.




\section[Preliminaries]{Preliminaries on independence relations}
\label{sec:preliminary}
Let $\monster$ be a monster model of some complete theory~$T$; 
``\intro{small}'' will mean 
``of cardinality smaller than the monstrosity of $\monster$''.

Let $\ind$ be a symmetric independence relation on $\monster$, in the sense of 
\cite{adler}; so $\ind$ is a ternary relation on small subsets of $\monster$ satisfying, for every small $A, B, C, D\subseteq \monster$, the following conditions. 
\begin{description}
\item[Invariance] for every $\sigma\in \aut(\monster)$,
$A\ind_BC\Rightarrow \sigma(A)\ind_{\sigma(B)}\sigma(C)$. 


\item[Normality] $A \ind_C B \Rightarrow A C \ind_C B$.

\item[Symmetry] $A\ind_BC\Rightarrow C\ind_BA$.

\item[Left Transitivity] (which we will simply call \textbf{Transitivity}) 
assuming $B\subseteq C\subseteq D$, $A\ind_B D$ iff $A\ind_BC$ and $A\ind_CD$.%
\footnote{In Adler's terminology, this axiom also includes Monotonicity and Base Monotonicity.}

\item[Finite character] 
$A\ind_BC$ iff $A_0\ind_BC$ for all finite $A_0\subseteq A$.

\item[Extension] 
there is some $A^\prime\equiv_BA$ such that $A^\prime\ind_BC$.

\item[Local character] 
there is some small $\kappa_0$ (depending only on~$\ind$),
such that there is some $C_0\subseteq C$ with 
$\card{C_0} < \kappa_0 + \card{A}^+$ and $A\ind_{C_0} C$. 

\end{description}

Following \cite{adler}, we say that $\ind$ satisfies \intro{strong
  finite character} if, whenever $A \notind_B C$, there exists a formula
$\phi(\x, \y, \z)$, $\av \in A$, $\bv \in B$, and $\cv \in C$, such that
$\monster \models \phi(\x, \y, \z)$, and, for every $\av' \in \monster$,
if $\monster \models \phi(\av', \bv, \cv)$, then $\av' \notind_B C$.

\medskip

We are \emph{not} assuming that $\monster$ eliminates imaginaries. For
non-small $A,B,C\subseteq\monster$, $A\ind_BC$ is defined to mean the
following: 
\begin{sentence}
For all small $A^\prime\subseteq A$, $B^\prime\subseteq B$, and small
$C^\prime\subseteq C$, there is some small $B''$ such that $B' \subseteq B''
\subseteq B$, and $A^\prime\ind_{B''}C^\prime$.
\end{sentence}
For every $A \subset \monster$ and $c \in \monster$, we say that $c \in
\mat(A)$ if $c \ind_A c$. 
By ``tuple'' we will mean a tuple of small length.

\begin{remark}\label{rem:local-char}
In the axioms of an independence relation, we can substitute the Local
Character Axiom with the following one:
\begin{sentence}[(LC')]
There is some $\kappa_0$ such that,
for every $\av$ finite tuple in $\monster$ and $B$ small subset of~$\monster$,
there exists $B_0 \subset B$, such that $\card{B_0} < \kappa_0$ and 
$\av \ind_{B_0} B$.
\end{sentence}
Moreover, the $\kappa_0$ in (LC') and the $\kappa_0$ in the Local Character
Axiom are the same.
\end{remark}
\begin{proof}
Let $A$ and $B$ be small subsets of~$\monster$.
For every $\av$ finite tuple in $A$, let $B_{\av} \subset B$, such that
$\card{B_{\av}} < \kappa_0$ and $\av \ind_{B_{\av}} B$.
Define $C := \bigcup \set{B_{\av}: \av \subset A \text{ finite}} \subseteq B$.
Then, by Finite Character, $A \ind_{C} B$, 
and $\card{C} < \kappa_0 + \card{A}^+$.
\end{proof}

\begin{remark}
Let $\av$ be a small tuple and $B$ and $C$ be small subsets of~$\monster$.
Assume that $\tp(\av / BC)$ is finitely satisfied in~$C$.
Then, $\av \ind_C B$.
\end{remark}
\begin{proof}
The assumptions imply that $\tp(\av/BC)$ does not fork in
the sense of Shelah's over~$C$.
By \cite[Remark 1.20]{adler}, we are done.
\end{proof}
\begin{question}
Is there a more direct proof of the above remark, that does not use Shelah's forking?
\end{question}

\begin{definition}
$\kappa_0$ is the smallest regular cardinal $\kappa$, 
such that, for every finite tuple
$\av$ and every small set $B$, there exists $B_0 \subseteq B$, with
$\card{B_0} < \kappa$ and $\av \ind_{B_0} B$.
\end{definition}

\begin{remark}
$\kappa_0 \leq \abs T^+$; hence, for every small sets $A$ and~$B$, there exists $B_0 \subseteq B$, 
such that $\card{B_0} \leq  \card T + \card A$ and $A \ind_{B_0} B$.
\end{remark}
\begin{proof}
Given a small tuple $\av$ and a small set $B$, let $C$ be a subset of
$\av$ such that $\tp(\av/B C)$ is finitely satisfied in~$C$, and
$\card{C} \leq \card{B} + \card T$ (\cite[Remark 2.4]{adler}).
Hence, $\av \ind_C B$.
\end{proof}

\begin{definition}
Let $A$ and $B$ be small subsets of $\monster$, with $A \subseteq B$.
Let $p$ be a type over $A$ and $q$ be a type over~$B$.
We say that $q$ is a \intro{nonforking extension} of~$p$, and write $q \nforkext p$, if $q$ extends $p$ and
$q \ind_A B$.
We say that $q$ is a \intro{forking extension} of~$p$, and write $q \forkext p$, if $q$ extends $p$ and
$q \notind_A B$.
\end{definition}

In the next lemma we use the fact that $\kappa_0$ is regular.

\begin{lemma}\label{lem:forking-chain}
Let $A \subset \monster$ be small and $p \in S_n(A)$. There is no sequence $p = p_0 \forkext p_1 \forkext \dots \forkext p_i \dots$, indexed by $\kappa_0$, of forking extensions of $p$.
\end{lemma}

\begin{proof}
Assume, for contradiction, that such a sequence exists.
For every $i < \kappa_0$, let $B_i$ be the domain of $p_i$.
Let $B := \bigcup_{i < \kappa_0} B_i$, and $D \subset B$ such that
$\card{D} < \kappa_0$ and $p \ind_D B$ 
($D$ exists by definition of $\kappa_0$).
Since $\kappa_0$ is regular, $D \subseteq B_i$ for some $i < \kappa_0$, and
therefore $p \ind_{B_i} B$, and in particular $p \ind_{B_i} B_{i+1}$, absurd.
\end{proof}

\begin{lemma}
All axioms for $\ind$, except extension, are valid also for large subsets 
of~$\monster$. 
The Local Character Axiom holds with the same $\kappa_0$.
\end{lemma}
\begin{proof}
Let us prove Local Character for non-small sets.
Let $A$ and $B$ be subsets of~$\monster$.
Assume, for contradiction, that, for every $C \subset B$, if $\card{C} <
\kappa_0$, then $A \notind C B$.
First, we consider the case when $A = \av$ is a finite tuple.
Define inductively a family of sets $(B_i: i < \kappa_0)$, such that:
\begin{enumerate}
\item 
each $B_i$ is a subset of $B$,
of cardinality strictly less than $\kappa_0$;
\item $(B_i: i < \kappa_0)$ is an increasing family of sets;
\item 
$\av \notind_{\bigcup_{j < i} B_j} B_{i}$.
\end{enumerate}
In fact, choose $B_0$ a finite subsets of~$B$,
such that  $\av \notind_{\emptyset} B_0$;
then, choose $B_i$ inductively, satisfying the given conditions.
Finally, let $p_i := \tp(\av / B_i)$.
Then, $(p_i: i < \kappa_0)$ is chain of forking extensions of
length~$\kappa_0$, contradicting Lemma~\ref{lem:forking-chain}.

If $A$ is infinite, proceed as in the proof of Remark~\ref{rem:local-char}.
\end{proof}

\begin{remark}\label{rem:localization}
Let $P$ be a (possibly, large) subsets of~$\monster$.
Let $\ind'$ the following relation on subsets of~$\monster$:
$A \ind'_C B$ iff $A \ind_{C P} B$.
Then, $\ind'$ satisfies 
Symmetry, Monotonicity, Base Monotonicity, Transitivity, Normality,
Finite Character and Local Character (with the same constant $\kappa_0$).
If moreover $P$ is $\aut(\monster)$-invariant (that is, $f(P) = P$ for every
$f \in \aut(\monster)$), then $\ind'$ also satisfies Invariance.
\end{remark}
\begin{proof}
Let us prove Local Character.
Let $A$ and $B$ be small subsets of~$\monster$.
Let $B_0 \subset B$ and $P_0 \subset P$, such that
$\card{B_0 \cup P_0} < \kappa_0 + \card{A}^+$ and
$A \ind_{B_0 P_0} B P$.
Then, $A \ind_{B_0 P} B$, and therefore $A \ind'_{B_0} B$.
\end{proof}

\begin{example}
Let $P \subset \monster$ be definable without parameters, and $\ind'$ be as in
the above remark.
It is not true in general that $\ind'$ is an independence relation: more
precisely, it might not satisfy the Extension axiom.
For instance, let $\pair{G, +}$ be a monster model of the theory of 
$\pair{\Q, +}$, and let
$\monster$ be the 2-sorted structure $G \sqcup G$, with the group structure on
the first sort, and the action by translation of the first sort on the second.
Notice that $\monster$ is $\omega$-stable; let $\ind$ be Shelah's forking
on~$\monster$, and $P$ be the first sort.
Choose $a$ and $b$ in the second sort arbitrarily.
Then, there is no $a' \elem a$ such that $a' \ind' b$, because
$\monster = \acl(Pb)$, but $\monster \neq \acl(P)$.
\end{example}

We will always consider models of $T$ as elementary substructures of $\monster$:
therefore, given $A \subset M \models T$, we can talk about $\mat(A)$ (a
subset of $\monster$, not of $M$!), and for $p \in S(M)$ we can define 
$p \ind_A M$. 

We say that $\ind$ is \intro{superior} if $\kappa_0 = \omega$.

\begin{remark}
$\ind$ is superior iff, for every finite set $A \subset \monster$ 
and every set $C
\subseteq \monster$, there exists a finite subset $C_0 \subset C$ such that
$A \ind_{C_0} C$.
Moreover, $\ind$ is superior iff $\nforkext$ is a well-founded partial
ordering, and therefore one can define a corresponding rank $\Uind$ 
for types (as in \cite[\S5.1]{wagner} for supersimple theories).
\end{remark}

\begin{example}\label{ex:independence-strict}
\begin{enumerate}
\item 
If $\monster$ is simple, we can take $\ind$ equal to Shelah's forking
$\ind[f]$.
$\kappa_0 \leq \card{T}^+$, and $\ind[f]$ is superior iff $\monster$ is
supersimple, and the rank induced by $\ind[f]$ is the Lascar rank~$\SU$.
\item
If $\monster$ is rosy, we can take $\ind$ equal to \th-forking
$\ind[\th]$.
$\ind[\th]$ is superior iff $\monster$ is superrosy.
\item
If $\monster$ is geometric, we can take $\ind$ equal to $\ind[$\acl$]$, the
independence relation induced by the algebraic closure.
$\ind$ is then superior of rank~1, and coincides with real \th-forking.

In all three cases, $\mat$ is the algebraic closure (that is, $\ind$ is \intro{strict}).
\end{enumerate}
\end{example}

\begin{example}\label{ex:independence-matroid}
If $\mat$ is an existential matroid on $\monster$, we can take $\ind$ equal to
the induced independence relation $\ind[$\mat$]$:
see \cite{fornasiero-matroids} for details; $\ind[$\mat$]$ is superior, of rank~1.
\end{example}

\begin{remark}
If $\monster$ is supersimple then $\ind$ is superior, and $\Uind \leq \SU$.
\end{remark}
\begin{proof}
Let $a \subset \monster$ be a finite tuple and $B \subset \monster$.
Let $B_0 \subset B$ be finite such that $a \ind[f]_{B_0} B$.
Then, by \cite[Remark~1.20]{adler}, $a \ind_{B_0} B$.
\end{proof}

Let $\Lang$ be the language of $T$.

Notice that the same structure $\monster$ could have more than one
independence relation.
The following example is taken from~\cite[Example~1.33]{adler}.
\begin{example}\label{ex:2}
Let $\Lang = \set{E(x,y)}$ and $\monster$ be the monster model where $E$ is an
equivalence relation with infinitely many equivalence classes, all infinite.
Then $\monster$ is $\omega$-stable, of Morley rank 2, and geometric.
Hence, $\monster$ has the following 2 independence relation: $\indf$ and
$\indacl$. 
More explicitly: $A \indacl_B C$ iff $A \cap C \subseteq B$, and
$A \indf_B C$ if $\acleq(AB) \cap \acleq(B C) = \acleq(B)$.
Both independence relations are superior,  strict, and satisfy Strong Finite
Character.
However, these two independence relations are different, and the rank of
$\monster$ is different according to the two different ranks:
$\U^{\acl}(\monster) = 1$ (the rank induced by $\acl$), while $\SU(\monster)
= 2$.
$\monster$ also has a third independence relation, which we will now describe.
For every $X \subset \monster$, let $\mat_E(X)$ be the set of elements which
are equivalent to some element in~$X$.
Define $A \ind[$E$]_C B$ if $\mat_E(AC) \cap \mat_E(BC) = \mat_E(C)$.
$\ind[$E$]$ is superior and satisfy Strong Finite Character, 
but it is not strict.
The rank of $\monster$ according to $\ind[$E$]$ is~$1$.
\end{example}

\begin{definition}
Let $A$ and $B$ be small subsets of~$\monster$ 
and $\pi(\x)$ be a partial type with parameters from~$A B$.
We say that $\pi$ does not fork over $B$ (and write $\pi \ind_B A$)
if there exists a complete type $q(\x) \in S_n(A B)$ containing $\pi$ 
and such that $q \ind_B A$  (or, equivalently, 
if there exists $\cv \in \monster^n$ realization of $\pi(\x)$ 
such that $\cv \ind_B A$).
\end{definition}
The interesting cases are when $\pi$ is either a single formula or a complete type over $A B$.

\begin{remark}
The above definition does not depend on the choice of the set of
parameters~$A$: that is, given some other subset $A'$ of $\monster$ 
such that the parameters of~$\pi$ are contained in $A' \cup B$, then
there exists $\cv \in \monster^n$ satisfying $\pi$
and such that $\cv \ind_B A$ iff
there exists  $\cv \in \monster^n$ satisfying $\pi$ and
such that $\cv \ind_B A'$.
\end{remark}
\begin{proof}
Assume that $\pi \ind_B A$.
Let $\cv \in \monster^n$ such that $\cv$ satisfies $\pi$ and
$\cv \ind_B A$.
Let $\cv' \in \monster^n$ such that $\cv' \elem_{A B} \cv$ and
$\cv' \ind_{AB} A'$. 
Therefore, $\cv'$ realizes $\pi$ and $\cv' \ind_B A$, and thus
$\cv' \ind_ B A'$.
\end{proof}


\subsection{The closure operator cl}

For this subsection, $A$, $A'$, $B$, and $C$ will always be subsets of
$\monster$, and $a$, $a'$ be elements of~$\monster$.

\begin{remark}[Invariance]\label{rem:mat-invariance}
Assume that $a \in \mat(B)$ and $a' B' \equiv a B$; then $a' \in \mat(B')$.
\end{remark}

\begin{remark}\label{rem:mat-ind}
If $a \in \mat(A)$ then $a \ind_A B$.
\end{remark}

\begin{proof}
Suppose $a\in \mat(A)$. 
By definition, $a \ind_A a$.
Moreover, $a \ind_{A a} B$.
Thus, by transitivity, $a \ind_A B$.
\end{proof}

\begin{remark}[Monotonicity]\label{rem:mat-monotone}
If $a \in \mat(B)$, then $a \in \mat(B C )$.
\end{remark}
\begin{proof}
By Remark~\ref{rem:mat-ind}, $a \ind_B B C a$.
Thus, $a \ind_{B C} a$.
\end{proof}

\begin{remark}
If $A \subseteq \mat(B)$, then $A \ind_B A$.
\end{remark}
\begin{proof}
\Wlog, $A = \pair{a_1, \dotsc, a_n}$ is a finite tuple.
By repeated application of Remarks~\ref{rem:mat-ind} and~\ref{rem:mat-monotone}, we have
$B A \ind_B B a_1$, $B A \ind_{B a_1} B a_1 a_2$, \dots.
Hence, by transitivity, $A \ind_B A$.
\end{proof}

\begin{remark}
If $A \subseteq \mat(B)$, then $A \ind_B C$.
\end{remark}
\begin{proof}
By the above remark, $A \ind_B A$.
Moreover, $A \ind_{AB} C$, and therefore, by transitivity, $A \ind_B C$.
\end{proof}

\begin{remark}\label{rem:mat-idempotent}
$A \ind_B C$ iff $A \ind_{\mat(B)} C$.
Therefore, $\ind_B$ and $\Uind(\cdot / B)$ depend only on $\mat(B)$, and not on~$B$.
\end{remark}

\begin{proposition}
$\mat$ is an invariant closure operator.
If moreover $\ind$ is superior, then $\mat$ is finitary.
\end{proposition}
\begin{proof}
The fact that $\mat$ is invariant is Remark~\ref{rem:mat-invariance}.
To prove that $\mat$ is a closure operator, we have to show:
\begin{description}
\item[Extension]
$A \subseteq \mat (A)$.\\
This is clear.
\item[Monotonicity]
$A \subseteq B$ implies $\mat(A) \subseteq \mat(B)$.\\
This follows from Remark~\ref{rem:mat-monotone}.
\item[Idempotency]
$\mat(\mat(A)) = \mat(A)$.\\
This follows from Remark~\ref{rem:mat-idempotent}.
\end{description}
Finally, we have to prove that if $\ind$ is superior, then $\mat$ is finitary.
Let $a \in \mat(B)$, that is $a \ind_B a$.
Since $\ind$ is superior, there exists $B_0 \subseteq B$ finite, such that $a \ind_{B_0} B$.
\end{proof}

\begin{conjecture}
$\mat$ is always finitary.
Moreover, $\mat$ is  definable: that is, for every $A \subseteq \monster$ small, the set
$\mat(A)$ is ord-definable over~$A$.
\end{conjecture}

\begin{remark}
If $A \ind_B C$, then $\mat(A B) \cap \mat(B C) = \mat(B)$ \rom(but the converse is not true in general).
\end{remark}


\section{Lovely pairs}

Let $P$ be a new unary predicate symbol, and
$\Ttwo$ be the theory of all possible expansions of $T$ to the language 
$\Ltwo := \Lang \cup \set P$.
We will use a superscript $1$ to denote model-theoretic notions for~$\Lang$,
and a superscript $2$ to denote those notions for $\Ltwo$: for instance, we
will write $a \equiv^1_C a'$ if the $\Lang$-type of $a$ and $a'$ over $C$ is
the same, or $S^2_n(A)$ to denote the set of complete $\Ltwo$-types in $n$ variables over~$A$.

Let $\M_P = \pair{M, P(M)} \models \Ttwo$.
Given $A \subseteq M$, we will denote $P(A) := P(M) \cap A$.
We will write $\Mp := \pair{\monster, P}$ for a monster model of $\Ttwo$.

\begin{definition}\label{def:lovely}
Fix some small cardinal $\kappa > \max(\kappa_0, \card{T})$;
we will say that $A$ is \intro{very small} if it is of cardinality much
smaller than~$\kappa$.
We say that $M_P$ is a $\ind$-lovely pair for $T$ (or simply a lovely pair if
$\ind$ and $T$ are clear from the context) if it satisfies the following
conditions:
\begin{enumerate}
\item 
$M$ is a small model of~$T$.
\item (Density property) Let $A \subset M$ be very small, and $q \in S^1_1(A)$
be a complete $\Lang$-1-type over $A$.  Assume that $q \ind_{P(A)} A$.  Then,
$q$ is realized in $P(M)$.
\item (Extension property) Let $A \subset M$ be very small and $q \in
S^1_1(A)$.  Then, there exists $b \in M$ realizing $q$ such that $b \ind_A P(M)$.
\end{enumerate}
\end{definition}
 
The above definition is the natural extension of  Definition~3.1
in~\cite{BPV} to the case when $\ind \neq \ind[f]$.
If we need to specify the cardinal~$\kappa$, we talk about $\kappa$-lovely pairs.

\begin{remark}
If $M_P$ satisfies the Extension property, then $M$ is $\kappa$-saturated (as
an $\Lang$-structure).
\end{remark}

\begin{definition}
Let $X \subseteq M \models T$.
The closure of $X$ in $M$ is 
\[
\mat^M(X) \coloneqq M \cap \mat(X).
\]
$X$ is \intro{closed} in $M$ if $\mat^M(X) = X$.
\end{definition}


\begin{remark}\label{rem:density-substructure}
If $M_P$ satisfies the Density property, then 
\begin{enumerate}
\item $P(M) \preceq M$;
\item
$P(M)$ is closed in $M$; 
\item
$P(M)$ is $\kappa$-saturated (as an $\Lang$-structure).
\end{enumerate}
\end{remark}
\begin{proof}

Let $\pv \in P(M)^n$ and $\phi(x, \y)$ be some $\Lang$-formula.
Assume that there exists $a \in M$ such that $M \models \phi(a, \pv)$.
Let $q := \tp^1(c/\pv)$.
Since $P(\pv) = \pv$, we have $q \ind_{P(\pv)} \pv$, and therefore there
exists $a' \in P(M)$ such that $M \models \phi(a', \pv)$, proving that
$P(M) \preceq M$.

Assume that $a \ind_{P(M)} a$.
Let $P_0 \subset P(M)$ be very small, such that $a \ind_{P_0} a$.
Let $q := \tp^1(a/ P_0 a)$; notice that $q \ind_{P_0} P_0 a$; therefore, by
the Density property, $q$ is realized in $P(M)$, that is $a \in P(M)$.

Assume that $A \subseteq P(M)$ is a very small subset and let
$q \in S^1_1(A)$.
Since $P(A) = A$, we have $q \ind_{P(A)} A$, and therefore $q$ is realized
in~$P(M)$, proving saturation. 
\end{proof}

\begin{example}
If $\ind$ is the trivial independence relation (that is, $A \ind_B C$ for
every $A$, $B$, and $C$), then $M_P$ is a lovely pair iff $M \models T$ is
sufficiently saturated and $P(M)= M$ (more precisely, the Extension Property
always hold, and the Density Property holds iff $M = P(M)$).
\end{example}

\begin{example}\label{ex:triple-love}
Let $T$ and $\monster$ be as in Example~\ref{ex:2}.
The 3 choices for an independence relation $\ind$  on $\monster$
will correspond to 3 different kind of lovely pairs; in particular, the theory
of lovely pairs for $T$ as a simple theory (which is equal to the theory of
Beautiful Pairs for $T$ \cite{poizat83}) is different from the theory of
lovely pairs for $T$ as a geometric theory. 

More explicitly, let $M_P \models \Ttwo$, and assume that $\M$ is sufficiently saturated.
Then, $M_P$ is a model of the theory of $\indf$-lovely pairs iff:
\begin{enumerate}
\item 
infinitely many equivalence classes are disjoint from~$P(M)$;
\item infinitely many equivalence classes intersect~$P(M)$;
\item for each equivalence class $C$, both $C \setminus P(M)$ and $C \cap
P(M)$ are infinite. 
\end{enumerate}
On the other hand, $M_P$ is a model of the theory of $\indacl$-lovely pairs
iff:
\begin{enumerate}
\item all equivalence classes intersect~$P(M)$;
\item for each equivalence class $C$, both $C \setminus P(M)$ and $C \cap P(M)$ are infinite.
\end{enumerate}
Finally, $M_P$ is a model of $\ind[$E$]$-lovely pairs iff:
\begin{enumerate}
\item 
infinitely many equivalence classes are disjoint from~$P(M)$;
\item infinitely many equivalence classes are contained in~$P(M)$;
\item every equivalence class is either disjoint or contained in~$P(M)$.
\end{enumerate}
\end{example}


\subsection{Existence}
\label{subsec:existence}

\begin{lemma}\label{lem:independent-model}
Let $A$, $B$, and $C$ be small subsets of~$\monster$.
Assume that $A \ind_C B$.
Then, there exists $N \prec \monster$ small model, such that
$A C \subseteq N$ and $N \ind_C B$.
In particular, for any $B$ and $C$ small subsets of~$\monster$, there exists
$N \prec \monster$ small model, such that $C \subseteq N$ and $N \ind_C B$.
\end{lemma}
\begin{proof}
Let $N' \prec N$ be a small model containing~$A$.
Choose $N \elem^1_{AC} N'$ such that $N \ind_C B$ ($N$ exists by the Extension Axiom).
\end{proof}
Modifying the above proof a little, we can see that, for any small
cardinal~$\kappa$, in the above lemma we can additionally require that $N$ is
$\kappa$-saturated, or that $\card N \leq \card A + \card T$, etc.

Fix $\kappa > \kappa_0$ a small cardinal.
\begin{lemma}\label{lem:lovely-existence}
$\kappa$-lovely pairs exist.
If $\pair{C, P(C)}$ is an $\Ltwo$-structure, where $P(C) \subseteq C \subset
\monster$, $C$~is small, and $P(C)$ is closed in~$C$ 
\rom(\ie, $\mat(P(C)) \cap C = P(C)$\rom), then there exists a $\kappa$-lovely pair
$M_P$, such that $\pair{C, P(C)}$ is an $\Ltwo$-substructure of $M_P$, and
$P(M) \ind_{P(C)} C$.
\end{lemma}
Therefore, given $A$ and $B$ small subsets of $\monster$, there exists
a $\kappa$-lovely pair $M_P$ such that $P(M) \ind_A B$ and $A \subseteq P(M)$.
\begin{proof}
The proof of~\cite[Lemma~3.5]{BPV} works also in this situation, with some
small modification.
Here are some details.
If $\pair{C, P(C)}$ is not given, define $C = P(C) = \emptyset$.
Construct a chain $\Pa{\pair{M_i, P(M_i)}: i < \kappa^+}$ of small
subsets of $\monster$, such that:

\begin{enumerate}[(a)]
\item $\pair{A, P(A)}$ is an $\Ltwo$-substructure of $\pair{M_0, P(M_0)}$, 
and $M_0 \ind_{P(A)} P(M_0)$;
\item
for every $i < \kappa^+$, any complete $\Lang$-type over $M_i$%
\footnote{Here we can choose: either in one variable, or in finitely many
  variables, or in a small number of variables: see \S\ref{sec:tuples}.}, 
which does not fork over $P(M_i)$ is realized in $P(M_{i+1})$; 
\item
$i < j < \kappa^+$ implies that $\pair{M_i, P(M_i)}$ is an $\Ltwo$-substructure
of $\pair{M_j, P(M_j)}$, and $P(M_j) \ind_{P(M_i)} M_i$;
\item
for $i$ successor, $M_i$ is a ($\kappa + \card{P(M_i)})^+$-saturated model
of~$T$.
\end{enumerate}

First, we construct $\pair{M_0, P(M_0)}$.
Let $M_0 \prec \monster$ be a small model containing $A$, and define
$P(M_0) := P(A)$.
Notice that $A \ind_{P(A)} P(M_0)$.

Given $\pair{M_i, P(M_i)}$, let
$\Pa{p_j: j < \lambda}$ to be an enumeration of all
$\Lang$-types over~$M_i$,%
\footnote{With the same meaning of ``type'' as in (b).}
such that $p_j \ind_{P(M_i)} M_i$.
Let
\begin{itemize}
\item 
$a_0$ be any realization of $p_0$ (in $\monster$);
\item
$a_1$ be a realization of $p_1$ such that 
$a_1 \ind_{M_i} a_0$;
\item \dots
\item
for every $j < \lambda$, let $a_j$ be a realization of $p_j$ such that
$a_j \ind_{M_i} \Pa{a_k: k < j}$.
\end{itemize}
Define $A := \Pa{a_j: j < \lambda}$.
It is then easy to see that $A \ind_{P(M_i)} M_i$.
Conclude the proof as in~\cite{BPV} (using Lemma~\ref{lem:independent-model} where necessary).
\end{proof}

\subsection[The back-and-forth argument]{Uniqueness: the back-and-forth argument}
\label{subsec:uniqueness}

In this subsection, $M_P$ will be a lovely pair.
The following definition is from~\cite{BPV}.
\begin{definition}
\begin{enumerate}
\item 
A set $A \subset M$ is \intro{\Pindependent{}} if $A \ind_{P(A)} P(M)$ (and
similarly for tuples).
\item
Given a possibly infinite tuple $\av$ from~$M$, $\Ptp(\av)$, the \Ptype 
of~$\av$, is the information which tell us which members of $\av$ are 
in~$P(M)$.
\end{enumerate}
\end{definition}

\begin{remark}
Let $A \subseteq M$ be \Pindependent.
Then $\mat(A) \cap P(M) = \mat^M  \Pa{P(A)}$.
\end{remark}
\begin{proof}
Let $c \in \mat(A) \cap P(M)$.
Then $c\ind_AP(M)$, by Remark 1.2, and $A \ind_{P(A)} P(M)$. Therefore $c\ind_{P(A)}P(M)$. Therefore $c\ind_{P(A)}c$, that is $c \in \mat\Pa{P(A)}$.
\end{proof}

\begin{proposition}
\label{prop:back-and-forth}
Let $M_P$ and $M'_P$ be two $\ind$-lovely pairs for~$T$.
Let $\av$ be a \Pindependent very small tuple from~$M$,
and $\av'$ be a \Pindependent tuple of the same length from~$M'$.
If $\av \elem^1 \av'$ and $\Ptp(\av) = \Ptp(\av')$, then
$\av \elem^2 \av'$.
\end{proposition}
\begin{proof}
Back-and-forth argument.
Let
\begin{multline*}
\Gamma := \bigl\{f: \av \to \av': \av \subset M,\ \av' \subset M',\
\av \et \av' \text{ very small},\ 
f \text { bijection},\\
\av \et \av' \text{ \Pindependent},\
\av \elem^1 \av',\
\Ptp(\av) = \Ptp(\av')\bigr\}.
\end{multline*}
We want to prove that $\Gamma$ has the back-and-forth property. Let $f:\bar{a}\rightarrow \bar{a}^\prime$ be in $\Gamma$. Let $b\in M$. 

\Case 1 Suppose $b\in P(M)$. Then, since $\bar{a}$ is \Pindependent,
$b\ind_{P(\bar{a})}\bar{a}$. Let $b^\prime\in M$ be such that
$\tp^1(\bar{a}b)=\tp^1(\bar{a}^\prime b^\prime)$. Then
$b^\prime\ind_{P(\bar{a}^\prime)}\bar{a}^\prime$. By density, we may assume
$b^\prime\in P(M)$. Then $\Ptp(\bar{a}b)= \Ptp(\bar{a}^\prime b^\prime)$. Since $b, b^\prime\in P(M)$, both $\bar{a}b$ and $\bar{a}^\prime b^\prime$ are \Pindependent. So the partial isomorphism has been extended.

\Case 2 Suppose $b\notin P(M)$. Let $\bar{f}$ be a very small tuple from $P(M)$ such that $\bar{a}b\bar{f}$ is \Pindependent. By repeated use of case 1, we obtain a small tuple $\bar{f}^\prime$ from $P(M')$ such that $\tp^1(\bar{a}\bar{f})=\tp^1(\bar{a}^\prime\bar{f}^\prime)$. Let $b^\prime \in M'$ be such that $\tp^1(\bar{a}b\bar{f})=\tp^1(\bar{a}^\prime b^\prime\bar{f}^\prime)$. By extension, we may assume $b^\prime\ind_{\bar{a}^\prime\bar{f}^\prime}P(M')$. It follows that $\bar{a}^\prime b^\prime \bar{f}^\prime$ is \Pindependent, and that $\mat(b' \av' \fv') \cap \mat(\av' \fv' P(M')) = \mat(\av' \fv')$.
Suppose $b^\prime\in P(M')$.
Then, $b' \in \mat(\av' \fv') \cap P(M')$.
Since $\av' \fv'$ is \Pindependent, $b' \in \mat(\fv' P(\av'))$.
Therefore, $b \in \mat^M \Pa{\fv P(\av)} \subseteq P(M)$, absurd.
So $\Ptp(\bar{a}b\bar{f}) = \Ptp(\bar{a}^\prime b^\prime\bar{f}^\prime)$. So the partial isomorphism has been extended.
\end{proof}

\begin{definition}
$\Td$ is the theory of $\ind$-lovely pairs (the ``$d$'' stands for ``dense'', a
legacy from the o-minimal case).
Given a property $\mathcal S$ of models of~$\Ttwo$,
we will say that ``$\mathcal S$ is first-order'' to mean that there is an
$\Ltwo$-theory $T'$ expanding~$\Ttwo$, such that every model of $\Ttwo$ with
the property $\mathcal S$ is a model of~$T'$, and every \emph{sufficiently
  saturated} model of $T'$ has the property~$S$.
In particular, by ``$\ind$-loveliness is first-order'' (or simply ``loveliness
is first-order'' when $\ind$ is clear from the context),  we will mean that
every sufficiently saturated model of $\Td$ is a lovely pair. 
\end{definition}

If $T$ is pregeometric, then $\indacl$-loveliness is first-order iff $T$ is
geometric (see \S~\ref{subsec:rank-one}).
\cite{BPV} and \cite{vassiliev05} investigate the question when $\indf$-loveliness is first-order for simple theories.
We will not study this question for~$\ind$, except for a partial result in 
Corollary~\ref{cor:low} .



\section[Near model completeness]{Near model completeness and other properties}
\label{sec:model-complete}

In this section we assume that loveliness is first order and that
$\Mp := \pair{\monster, P}$ is a monster model of~$\Td$.

\subsection{Near model completeness}

\begin{lemma}\label{lem:local-large}
For every finite tuple $\bar{a}$ from $\monster$ there is some small $C\subseteq P$ such that $\bar{a}\ind_{C^\prime}P$ for all $C^\prime \equiv^1_{\bar{a}}C$ such that $C^\prime\subseteq P$. 
\end{lemma}

\begin{proof}
Suppose not. Then for each small $C_{\alpha}\subseteq P$ there is some small
$C_{\alpha+1}\subseteq P$ such that $\bar{a}\notind_{C_\alpha}C_{\alpha+1}$
and $\tp^1(C_{\alpha}/\bar{a})$ is realised in $C_{\alpha+1}$. 
By compactness there is, for any $\kappa$ less than the monstrosity of
$\monster$, an increasing sequence $(C_{\alpha})_{\alpha<\kappa}$ such that
$\bar{a}\notind_{C_{\alpha}}C_{\beta}$ for all $\alpha<\beta<\kappa$,
contradicting Lemma~\ref{lem:forking-chain}.
\end{proof}

\begin{definition}
Given a tuple of variables~$\z$, we will write $P(\z)$ as a shorthand for $P(z_1) \wedge \dots \wedge P(z_n)$.
A \intro{special} formula is a formula of the form
$(\exists \z) \Pa{P(\z) \wedge \varphi(\x,\z)}$, where $\varphi$ is an $\Lang$-formula. 
\end{definition}

\begin{proposition}[Near model completeness]
\label{prop:model-complete}
Every $\Ltwo$-formula without parameter is equivalent modulo $T^2$ to a Boolean combination of special formulae \rom(without parameters\rom).
\end{proposition}

\begin{proof}
Let $\bar{a}$ and $\bar{b}$ be finite tuples from~$\monster$. Suppose they both satisfy exactly the same special formulae.
It suffices to prove that $\tp^2(\bar{a}) = \tp^2(\bar{b})$. 
Let $C \subset P$ be as in Lemma~\ref{lem:local-large}.
By assumption and compactness there is some
$D\subseteq P$ such that $\tp^1(\bar{a}C) = \tp^1(\bar{b}D)$. Suppose
$\bar{b}\notind_D P$. Then there is some small $E\subseteq P$ such that 
$\bar{a}\notind_D E$. By assumption and compactness
there exist $C^\prime, E^\prime \subseteq P$ such that
$\tp^1(\bar{a}C^\prime E^\prime) = \tp^1(\bar{b}DE)$. But then $C^\prime\models
\tp^1(C/\bar{a})$ and so $\bar{a}\ind_{C^\prime}P$ and hence
$\bar{a}\ind_{C^\prime}E^\prime$. This contradicts invariance of~$\ind$. So
$\bar{b}\ind_D P$. So both $\bar{a}C$ and $\bar{b}D$ are
\Pindependent. It follows from our assumption that $\Ptp(\bar{a}C) =
\Ptp(\bar{b}D)$ and we know $\tp^1(\bar{a}C) = \tp^1(\bar{b}D)$.  
Therefore $\tp^2(\bar{a}C) = \tp^2(\bar{b}D)$. 
\end{proof}


\subsection{Definable subsets of P}

\begin{lemma}[{\cite[4.1.7]{boxall}}]
\label{lem:trace}
Let $\bv \in \monster^n$ be \Pindependent.
Given a set $Y \subseteq P^m$, \tfae:
\begin{enumerate}
\item $Y$ is $\Td$-definable over $\bv$;
\item there exists $Z \subseteq \monster^m$ that is $T$-definable over~$\bv$, such that
$Y = Z \cap P^m$.
\end{enumerate}
\end{lemma}

\begin{proof}
$(2 \Rightarrow 1)$ is obvious.
Assume (1).
Then, (2) follows from compactness and the fact that the $\Ltwo$-type over $\bv$ elements from $P$ is determined by their $\Lang$-types.
\end{proof}


\subsection{Definable and algebraic closure}

Let $\pair{M, P(M)}$ be a lovely pair.
In the next proposition we will consider imaginary elements; to simplify the notation, we will use $\acl^1$ for the algebraic closure for imaginary elements in~$\Meq$, and $\acl^2$ for the algebraic closure for imaginary elements in $\pair{M, P(M)}^\eq$, and similarly for~$\dcl$.

\begin{proposition}[{\cite[4.1.8, 4.1.9]{boxall}}]
\label{prop:acl}
Let $\av \subset M$ be \Pindependent.
Then, $\acl^2(\av) \cap \Meq = \acl^1(\av) \cap \Meq$ and
$\dcl^2(\av) \cap \Meq = \dcl^1(\av) \cap \Meq$.
\end{proposition}

\begin{proof}
Clearly, $\dcl^2(\av) \cap \Meq \subseteq \dcl^1(\av) \cap \Meq$, and similarly for~$\acl$.
We have to prove the opposite inclusions.
\Wlog, $\av$ is a very small tuple.

Let us prove first the statement for~$\dcl$.
So, let $e \in \Meq \cap \dcl^2(\av)$.
Let $\bv$ be a \Pindependent very small tuple in $M$ containing~$\av$, such that $e \in \dcl^1(\bv)$.
Denote $\bv_0 \coloneqq P(\bv)$ and $\bv_1 \coloneqq \bv \setminus P(\bv)$; notice that
$\bv \ind_{\bv_0} P(M)$.

\begin{claim}\label{cl:dcl-1}
Let $\bv' \elem^1_{\av} \bv$ be such that $\bv' \ind_{\av} \bv$.
Then, $\bv' \elem^1_e \bv$.
\end{claim}
Notice that $\bv'_0 \nsubseteq P(M)$ in general.
Since $a$ is \Pindependent, we have $a \ind_{P(\av)} \bv_0$, and therefore $\bv'_0 \ind_{P(\av)} \av$.
Since moreover $\bv'_0 \ind_{\av} \bv$, by transitivity we have $\bv'_0 \ind_{P(\av)} \bv$, and thus
$\bv'_0 \ind_{P(\bv)} \bv$.
Therefore, by the Density property, there exists $\bv''_0 \subset P(M)$ such that $\bv''_0 \elem^1_{\bv} \bv'_0$.

Let $\theta \in \aut^1(\monster / \bv)$ be such that $\theta(\bv'_0) = \bv''_0$.
Let $r(x) \coloneqq \tp^1(\bv'_1 / \bv \bv'_0)$, and $q \coloneqq \theta(r) \in S^1(\bv \bv''_0)$.
By the Extension property, there exists $\bv''_1$ in~$M$, 
such that $\bv''_1$ realizes~$q$ and
\begin{equation}\label{eq:dcl}
\bv''_1 \ind_{\bv \bv''_0} P(M).
\end{equation}
Denote $\bv'' \coloneqq \bv''_0 \bv''_1$; notice that
$\bv'' \elem^1_{\bv} \bv'$.
Since $\bv \ind_{\av} \bv'$, we have $\bv \ind_{\av} \bv''$, and therefore
$\bv_1'' \ind_{\bv_0'' \av} \bv \bv''_0$.
Therefore, by \eqref{eq:dcl} and transitivity, $\bv''_1 \ind_{\av \bv''_0} P(M)$.
Since moreover $\av$ is \Pindependent, by transitivity again we have
$\bv'' \ind_{\bv''_0} P(M)$.

Thus, we proved that $\bv''$ is \Pindependent and has the same \Ptype over $\av$ as~$\bv$, and thus
$\bv'' \elem^2_{e} \bv$.
Since moreover, by definition, $\bv'' \elem^1_{\bv} \bv'$, we have
$\bv'' \elem^1_e \bv'$, and the claim is proved.

Since $e \in \dcl^1(\bv)$, there exists a function $f$ which is $T$-definable without parameters, and such that
$e = f(\bv)$

\begin{claim}\label{cl:dcl-2}
Assume that $\bv'' \subset M$ and
$\bv'' \elem_{\av} \bv$.
Then, $e = f(\bv'')$.
\end{claim}
It follows immediately from Claim~\ref{cl:dcl-1}.

Now, since loveliness is first-order, we can assume that $\pair{M, P(M)}$ is sufficiently saturated; therefore,
Claim~\ref{cl:dcl-2} implies that $e \in \dcl^1(\av)$, and thus
$\dcl^1(\av) \cap \Meq = \dcl^2(\av) \cap \Meq$.

Assume now that $e \in \acl^2(\av) \cap \Meq$.
Let $X$ be the set of realizations of $\tp^2(e / \av)$.
Notice that $X$ is a finite subset of~$\Meq$, and therefore it is definable in~$\Meq$;
Let $e'$ be a canonical parameter for $X$ in the sense of~$\Meq$.
Since $e' \in \dcl^2(\av) \cap \Meq$, by the first assertion we have $e' \in \dcl^1(\av)$.
Since $e \in \acl^1(e')$, we $e \in \acl^1(\av)$.
Therefore, $\acl^1(\av) \cap \Meq = \acl^2(\av) \cap \Meq$.
\end{proof}

\begin{proposition}[{\cite[6.1.3]{boxall}}]
Let $\av \subset M$ be \Pindependent.
Let $f: P(M)^n \to \Meq$ be $\Td$-definable with parameters~$\av$.
Then, there exists $g: M^n \to \Meq$ which is $T$-definable with parameters~$\av$,
such that $f = g \rest P(M)^n$.
\end{proposition}
\begin{proof}
Let $\pair{N, P(N)}$ be an elementary extension of $\pair{M, P(M)}$ and $c \in P(N)$.
By Proposition~\ref{prop:acl}, $f(c) \in \dcl^1(\av)$, and therefore there exists 
$g_i: N^n \to N^\eq$ which is $T$-definable with parameters~$\av$, such that $f(c) = g_i(c)$.
By compactness, finitely many $g_i$ will suffice.
The conclusion follows from Lemma~\ref{lem:trace}.
\end{proof}

Notice that in the above proposition we were not able to prove the stronger
result that $\param g \in \dcl^2(\param f)$,
where $\param f$ is the canonical parameter of $f$ according to~$\Td$ (\cf \cite[6.1.3]{boxall}).
Nor were we able to prove any form of elimination of imaginaries for $\Td$ (\cf \cite[theorems 1.2.4, 1.2.6 and 1.2.7]{boxall} and \cite[\S8.5]{fornasiero-matroids}).


\section{Small and imaginary tuples }\label{sec:tuples}

\subsection{Small tuples}
We show how in Definition~\ref{def:lovely}, we can pass from $\Lang$-1-types to
$\Lang$-types in very small number of variables.
Let $M_P = \pair{M, P(M)}$ be a small model of $\Ttwo$.

\begin{lemma}\label{lem:multi-density}
Assume that $M_P$ 
satisfies the Density Property \rom(for $\Lang$-1-types\rom).
Then, $M_P$ satisfies the Density Property for $\Lang$-types of very
small length.
\end{lemma}
\begin{proof}
Let $A \subset M$ be very small and $\bv$ be a tuple in $M$ of very small
length, such that $\bv \ind_{P(A)} A$.
We must prove that there exists $\bv' \subset P(M)$ such that 
$\bv' \elem^1_A \bv$.
Define
\[
\Ffam := \set{f \text{ partial 1-automorphism of } M/A:\
\Dom(f) \subseteq \bv,\
\Imag(f) \subseteq P(M)}.
\]
Let $f \in \Ffam$ be a maximal element (Zorn).
I claim that $\Dom(f) = \bv$  (this suffices to prove the conclusion).
Assume not.
Let $\dv := \Dom(f)$ and $e \in \bv \setminus \dv$.
Let $\dv' := f(\dv) \subset P$, $q := \tp^1(e/ A \dv)$, and $q' := f(q)$.
Notice that $e' \dv' \ind_{P(A)} A$, and therefore
$q' \ind_{P(A) \dv'} A \dv'$.
Since $\dv' \subseteq B$, the Density Property for $\Lang$-1-types implies
that there exists $e' \in P(M)$ satisfying $q'$,
and hence $\dv' e'' \elem^1_A \dv e$, contradicting the maximality of~$f$.
\end{proof}

\begin{lemma}
Assume that $M_P$ satisfies the Extension Property \rom(for $\Lang$-1-types\rom).
Then, $M_P$ satisfies the Extension Property for $\Lang$-types of very
small length.
\end{lemma}
\begin{proof}
Let $A \subset M$ be very small and $\bv$ be a tuple in $M$ of very small
length.
We must prove that there exists $\bv' \subset M$ such that 
$\bv' \elem^1_A \bv$ and $\bv' \ind_A P$.
Define
\[
\Ffam := \set{f \text{ partial 1-automorphism of } M/A:\ 
\Dom(f) \subseteq \bv,\
\Imag(f) \ind_A P(M)}.
\]
Let $f \in \Ffam$ be a maximal element (Zorn).
I claim that $\Dom(f) = \bv$  (this suffices to prove the conclusion).
Assume not.
Let $\dv := \Dom(f)$ and $e \in \bv \setminus \dv$.
Let $\dv' := f(\dv)$, $q := \tp^1(e/ A \dv)$, and $q' := f(q)$.
The Extension Property for $\Lang$-1-types implies
that there exists $e' \in M$ satisfying $q'$ and such that
$e' \ind_{A \dv'} P(M)$.
Besides, by assumption, $\dv' \ind_A P(M)$; therefore
$e' \dv' \ind_A P$.
Since moreover $e' \dv' \elem^1_A e \dv$, we have  a contradiction.
\end{proof}


\subsection{Imaginary tuples}

Given an independence relation $\ind$ on $\monster$, we do not know if there
always exists an independence relation $\indeq$ on $\monstereq$
extending~$\ind$. 
However, as the next lemma shows, the independence relation $\indeq$, if it exists, it is unique.%
\footnote{Thanks to H.~Adler for pointing this out.}%

\begin{lemma}
Let $\indeq$ be an independence relation on $\monstereq$
extending~$\ind$.
Let $A$, $B$, and $C$ be small subsets of~$\monstereq$.
Then, the following are equivalent:
\begin{enumerate}
\item
$A \indeq_C B$;
\item
There exist $A_0$ and $B_0$ small subsets of~$\monster$, such that 
$A \subseteq  \dcleq(A_0)$,
$B \subseteq  \dcleq(B_0)$, and
for every $C_0$ small subset of $\monster$ with $C \subseteq \dcleq(C_0)$,
there exists $A_0' \subset \monster$ such that
$A_0' \elem_{B_0 C} A_0$ and $A_0' \ind_{C_0} B_0$.
\end{enumerate}
\end{lemma}
\begin{proof}
Exercise: first reduce to the case when $A$ and $B$ are subsets of~$\monster$.
\end{proof}

Remember that $\ind$ is strict if $\mat = \acl$.
\begin{warning*}
It may happen that $\ind$ is strict, but $\indeq$ is not strict:
consider for instance the independence relation $\indacl$ in Example~\ref{ex:2}
(\cf \cite[\S6]{fornasiero-matroids}).
\end{warning*}

\begin{proviso*}
For the remainder of this section, 
we assume that $\indeq$ is an independence relation on $\monster^{\eq}$
extending~$\ind$.
Moreover, we also assume that $P$ is closed in~$\monster$.
\end{proviso*}

\begin{definition}
Assume that $M_P = \pair{M, P(M)}$ is a model of $\Ttwo$.
Let $A \subseteq \Meq$.
\begin{enumerate}
\item
Define $P(A) := \mateq(P) \cap A$. 
Notice that $P(A)$ is the same as before 
if $A$ is a set of \emph{real} elements.
\item
We say that $A$ is \Pindependent if $A \indeq_{P(A)} P$.  
\item
Given a possibly infinite tuple $\av$ from~$\Meq$, $\Ptp(\av)$, the \Ptype
of $\av$, is the information which tell us which members of $\av$ are in
$P(\Meq)$ (again, notice that $\Ptp(\av)$ is the same as before 
if $\av$ is a tuple of real elements).
\end{enumerate}
\end{definition}

The following two lemmas are the analogues of \cite[Remark~3.3]{BPV}.

\begin{lemma}\label{lem:imaginary-density}
If $M_P$ satisfies the Density Property \rom(for $\Lang$-1-types\rom), then 
it satisfies the Density Property for imaginary tuples:
that is, if $\cv$ and $\av$ are very small tuples in
$\Meq$ and $\cv \indeq_{P(\av)} \av$, then there exists $\cv' \in \mateq(P)$, 
such that $\cv' \elem^1_{\av} \cv$.
\end{lemma}
\begin{proof}
Fix $\dv$ a small tuple in $M$ such that
$\cv = [\dv]_E$, for some $\emptyset$-definable equivalence relation~$E$.
Let $\av_0 := P(\av)$ and $\av_1 := \av \setminus P$.
Fix $\bv_1$ small tuple in $M$ such that 
$\av_1 = [\bv_1]_F$, for some 
$\emptyset$-definable equivalence relation~$F$. 
Let $\dv'$ in $M$ such that $\dv' \elem^1_{\av_0 \cv} \dv$ and
$\dv' \indeq_{\av_0 \cv} \av$. 
Since, by assumption, $\cv \indeq_{\av_0} \av$, we have
$\dv' \indeq_{\av_0} \av$.
Notice, moreover, that $[\dv']_E = \cv$.
Let $\bv'_1 \elem^1_{\av} \bv_1$ such that 
$\bv_1' \indeq_{\av} \dv'$; notice that $[\bv'_1]_F = \av_1$.
Moreover, by Transitivity, $\bv'_1 \indeq_{\av_0} \dv'$; 
therefore, $\av_0 \bv'_1 \indeq_{P(\av_0 \bv'_1)} \dv'$. 
Let $\bv_0$ be a small tuple in $P$ such that $\av_0 \in \acleq(\bv_0)$;
moreover, we can choose $\bv_0$ that satisfies 
$\bv_0 \indeq_{\av_0} \bv_1' \dv'$.
Hence, $\bv_1' \ind_{\bv_0} \dv'$: notice that all tuples are real. 
Hence, by Lemma~\ref{lem:multi-density}, there exists $\dv''$ in $P$
such that $\dv'' \elem^1_{\bv_0 \bv'_1} \dv'$.
Define $\cv'' := [\dv'']_E$.
Then, $\cv'' \elem^1_{\av} \cv$ and $\cv'' \in \mateq(P)$.
\end{proof}

\begin{lemma}\label{lem:imaginary-extension}
If $M_P$ satisfies the Extension Property \rom(for $\Lang$-1-types\rom), then 
it satisfies the Extension Property for imaginary tuples:
that is, if $\cv$ and $\av$ are very small tuples in $\Meq$,
then there exists $\cv' \in M^{\eq}$, 
such that $\cv' \elem^1_{\av} \cv$ and $\cv' \indeq_{\av} P$.
\end{lemma}
\begin{proof}
Fix $\bv$ and $\dv$ small tuples in $M$ such that
$\av = [\bv]_F$ and $\cv = [\dv]_E$, for some $\emptyset$-definable
equivalence relations $E$ and~$F$.
Let $\bv' \elem^1_{\av} \bv$ such that $\bv' \indeq_{\av} \dv$.
Notice that $\av = [\bv']_F$ and that $\av \in \mateq(\bv')$ (because
$\av \in \dcleq(\bv')$).
By Lemma~\ref{lem:multi-density}, there exists $\dv'$ in $M$ such that
$\dv' \elem^1_{\bv'} \dv$ and $\dv' \ind_{\bv'} P$.
Notice that $\dv' \elem_{\bv' \av} \dv$; therefore, since 
$\dv \ind_{\av} \bv'$, we have $\dv' \indeq_{\av} \bv'$, hence, by
transitivity, $\dv' \ind_{\av} P$.
Finally, let $\cv' := [\dv']_E$.
Thus, $\cv' \elem^1_{\av} \cv$ and $\cv \indeq_{\av} P$.
\end{proof}

\begin{lemma}\label{lem:imaginary-type}
Let $M_P$ and $M'_P$ be two $\ind$-lovely pairs for~$T$.
Let $\av$ be a \Pindependent small tuple from~$\Meq$, 
and $\av'$ be a \Pindependent tuple of the same length and the same sorts
from~${M'}^{\eq}$.
If $\av \elem^1 \av'$ and $\Ptp(\av) = \Ptp(\av')$, then
$\av \elem^2 \av'$.
\end{lemma}
\begin{proof}
Again, by a back-and-forth argument.
Denote $P := P(M)$ and $P' := P(M')$.
Let $\av_0 := P(\av)$ and $\av'_0 := P(\av')$.
Let
\begin{multline*}
\Gamma := \bigl\{f: \av \to \av': \av \subset \Meq,\ \av' \subset \Mpeq,\
\av \et \av' \text{ very small},\ 
f \text { bijection},\\
\av \et \av' \text{ \Pindependent{}},\
\av \elem^1 \av',\
\Ptp(\av) = \Ptp(\av')\bigr\}.
\end{multline*}
We want to prove that $\Gamma$ has the back-and-forth property.
So, let $f: \av \to \av'$ be in $\Gamma$, and $\cv \subset\Meq \setminus \av$
be a small tuple;
we want to find $g \in \Gamma$ such that $g$ extends $f$ and $\cv$ is
contained in the domain of~$g$.
We can reduce ourselves to two cases.

\Case 1 $\cv \subset \mateq(P)$.
Let $q := \tp^1(\cv/\av)$ and $q' := f(q)$.
Since $\av$ is \Pindependent, we have $q \indeq_{\av_0} \av$, and hence
$q' \indeq_{\av'_0} \av'$.
Therefore, by the Density property in Lemma~\ref{lem:imaginary-density}, 
there exists $\cv' \subset \mateq(P')$ satisfying $q'$;
extend $f$ to $\av \cv$ setting $f(\cv) = \cv'$.

\Case 2 $\cv \subset \Meq \setminus \mateq(P)$.
Let $\pv_0$ be a small subset of $P$ such that
$\cv \pv_0 \av$ is \Pindependent.
By Case~1, \wloG $\pv_0 \subseteq \av_0$, \ie $\cv \av$ is \Pindependent,
that is $\cv \av \ind_{\av_0} P$.
Let $q := \tp^1(\cv/\av)$ and $q' := f(q)$.
By Lemma~\ref{lem:imaginary-extension},
there exists $\cv' \subset \Mpeq$ satisfying $q'$ such that 
$\cv'\indeq_{\av'} P'$. 
\begin{claim}
$\cv' \cap \mateq(P') = \emptyset$.
\end{claim}
Assume, for contradiction, that $c_0 \in \cv' \cap \mateq(P')$.
Since $\cv' \indeq_{\av'} P'$, we have
$c_0 \in \mateq(\av') \cap \mateq(P') = \mat(\av'_0)$
hence, $c_0 \in \mat(\av) \subseteq \mateq(P)$, absurd.

Thus, $\cv \av$ and $\cv' \av'$ have the same \Ptype and the same
$\Lang$-type.
Moreover, by transitivity, $\cv' \av' \indeq_{\av'_0} P'$, that is
$\cv' \av'$ is \Pindependent.
Thus, we can extend $f$ to $\cv \av$ setting $f(\cv) = \cv'$.
\end{proof}

\section{Lowness and equivalent formulations of loveliness} 
\label{sec:low}

Let $M_P = \pair{M, P(M)} \models \Ttwo$.

\begin{remark}
The Density Property for $M_P$ is equivalent to:
\begin{sentence}[(\S)]
For every $A$ very small subset of~$M$ and $q \ind S^1_1(A)$, 
if $q \ind_{P(M)} A$, then $q$ is realized in~$P(M)$.
\end{sentence}
\end{remark}
\begin{proof}
Let $A \subset M$ be very small.

The proof that the Density Property implies (\S) 
is as in~\cite[Remark~3.4]{BPV}: given
$c \in \monster$ such that $c \ind_{P(M)} A$,
let $P_0 \subset P(M)$ very small such that  $P(M) \ind_{P_0} A$;
by Transitivity, $c \ind_{P(A) P_0} A$, and therefore, by the Density Property,
there exists $c' \in P(M)$ such that \mbox{$c' \elem^1_{P_0 A} c$}.

For the converse, assume that $\pair{M, P(M)}$ satisfies (\S),
and let $q \in S^1_1(A)$ such that $q \ind_{P(A)} A$.
Let $r \in S^1_1(A \cup P(M))$ be a non-forking extension of~$q$.
Let $c \in \monster$ be a realization of~$r$.
By transitivity, we have that $c \ind_{P(A)} A P(M)$, and therefore
$c \ind_{P(M)} A$.
Hence, by (\S), there exists $c' \in P(M)$ such that 
$c \equiv^1_A c$, and hence $c'$ is a realization of $q$ in $P(M)$.%
\footnote{Thanks to E.~Vassiliev for the proof.}
\end{proof}

\begin{lemma}\label{lem:density-saturated}
Assume that $M_P$ is $\kappa$-saturated. Then, \tfae:
\begin{enumerate}
\item $M_P$ satisfies the Density property;
\item for every $A \subset M$ very small, for every $\Lang$-formula $\phi(x)$
in 1 variable with parameters from~$A$,
if $\phi$ does not fork over $P(A)$, then $\phi$ is realized in $P(M)$.
\item for every $A \subset M$ very small, for every $\Lang$-formula $\phi(\x)$
in many variables with parameters from~$A$,
if $\phi$ does not fork over $P(A)$, then $\phi$ is realized in $P(M)$.
\end{enumerate}
\end{lemma}
\begin{proof}
$(1 \Rightarrow 2)$.
Let $q(x) \in S^1_1(A)$ extending $\phi(x)$ such that $q \ind_{P(A)} A$.
Choose $c \in P(M)$ realizing $q(x)$.
(Notice that we did not use the fact that $M_P$ is $\kappa$-saturated).

$(2 \Rightarrow 1)$.
Let $q \in S^1_1(A)$ be a complete $\Lang$-1-type over some very small set
$A$, such that $q \ind_{P(A)} A$.
Consider the following partial $\Lang^2$-1-type over~$A$:
$\Phi(x) := q(x) \et x \in P$.
By (2), $\Phi$ is consistent, and hence, by saturation, realized in $P(M)$.

$(3 \Rightarrow 2)$ is trivial, and $(1 \Rightarrow 3)$ follows as in
$(1 \Rightarrow 2)$ using Lemma~\ref{lem:multi-density}.
\end{proof}

\begin{definition}
Let $\x$ and $\y$ be finite tuples of variables.
Fix a formula $\phi(\x, \y)$.  Let $\z$ be very small tuple, and define
\[
\Sigma_{\phi, \z}(\y, \z) := \set{\pair{\bv, \cv}: \phi(\x, \bv) \text{ forks over } \cv}.
\]
We say that $\ind$ is \intro{low} if $\Sigma_\phi(\y, \z)$ is type-definable,
for every formula $\phi$.
\end{definition}

\begin{remark}
When $T$ is simple and $\indf$ is Shelah's forking, then $\indf$ is low iff $T$ is a low simple theory.
Moreover, if $T$ is stable, then $T$ (and hence~$\indf$) is low.
See \cite{BPV} for definitions and proofs.
\end{remark}

\begin{corollary}[{\cite[4.1]{BPV}}]\label{cor:low}
If $\ind$ is low iff the Density property is first order.
If $\ind$ is low, the axiomatization for the Density property is:
\begin{equation}
\label{eq:density}
(\forall \bv)\ (\forall \cv \in P)\
\Pa{\pair{\bv, \cv} \notin \Sigma_{\phi, \z}(\y, \z) \Rightarrow 
(\exists a \in P)\ \phi(a, \bv, \cv)},
\end{equation}
where $\phi(x, \y)$ varies among all the $\Lang$-formulae, with $x$ a single
variable, $\y$ and $\bv$ are finite tuples of variables of the same length, 
and $\z$ and $\cv$ are very small tuples of variables of the same length.
\end{corollary}
Notice that if $\ind$ is low, then \eqref{eq:density} is indeed given by
a set of axioms.
\begin{proof}
Assume that $\ind$ is low.
Let $M_P$ be a $\kappa$-saturated model of $\Ttwo$.
We have to prove that $M_P$ satisfies the Density property iff
\eqref{eq:density} holds.
Notice that \eqref{eq:density} is equivalent to:\\
``If $\phi(x, \bv)$ does not fork over $\cv$, with $\cv$ very small tuple 
in $P(M)$, then $\phi(x, \bv)$ is satisfied in~$P(M)$''.\\
By Lemma~\ref{lem:density-saturated}, this is equivalent to the Density
property.

Conversely, assume that the Density property is first order.
Fix an $\Lang$-formula $\phi(\x, \y)$; we must show that
$\Sigma := \Sigma_{\phi, \z}(\y, \z)$ is preserved under ultraproducts.
Assume that, for every $i$ in some index set $I$,
$\pair{\bv_i, \cv_i} \in \Sigma$, that is
$\phi(\x, \bv_i)$ forks over $\cv_i$ (where $\length{\bv_i} = \length{\y}$,  
and $\cv_i$ are very small tuples all of the same length).
By Lemma~\ref{lem:lovely-existence}, for every $i \in I$ there exists
a lovely pair $\pair{N_i, P_i}$ such that $\cv_i \in P_i$ and
$\bv_i \ind_{\cv_i} P_i$.
Let $\pair{N, P(N), \bv, \cv}$ be an ultraproduct of the 
$\pair{N_i, P_i, \bv_i, \cv_i}$, and let
$\pair{M, P}$ be a $\kappa$-saturated elementary extension of
$\pair{N, P(N)}$.
\begin{claim}
$\phi(\x, \bv)$ is not realized in~$P$.
\end{claim}
In fact, for every $i \in I$, since   $\phi(\x, \bv_i)$ forks over $\cv_i$,
and $\pair{N_i, P_i}$ satisfies the Density property, $\phi(\x, \bv_i)$ is not
realized in~$P_i$; thus, $\phi(\x, \bv)$ is not realized in $P(N)$, and hence
not in~$P$.

Moreover, since the Density property is first order, $\pair{M, P}$ satisfies
it.
Therefore, by Lemma~\ref{lem:density-saturated}, $\phi(\x, \bv)$ forks over~$\cv$.
\end{proof}


\subsection{The rank 1 case}
\label{subsec:rank-one}


In this subsection we will study more in details the case when $\ind$ is superior and $\Uind(\monster) = 1$.

\begin{remark}
Let $\monster$ be pregeometric (that is, $\acl$ has the Exchange property).
Then, $\indacl\,$ is superior and $\Uacl(\monster) = 1$.
If $V \subseteq \monster$ is definable, then $V$ is infinite iff $\Uacl(V) = 1$.
\end{remark}

\begin{proviso*}
For the remainder of this subsection, $\ind$ is superior and $\Uind(\monster) = 1$.
Moreover, we denote $\U \coloneqq \Uind$.
Finally, $M$~is a small model and $M_P = \pair{M, P(M)} \models \Ttwo$.
\end{proviso*}

\begin{lemma}
If $M_P$ satisfies the Density Property, then
\begin{enumerate}
\item $P(M)$ is an elementary substructure of~$M$;
\item 
$P(M)$ is $\kappa$-saturated;
\item 
$P(M)$ is \intro{$\mat$-dense} in $M$: that is,
for every $T$-definable subset $V$ of $M$, if $\U(V) = 1$, then $P(M) \cap V
\neq \emptyset$. 
\end{enumerate}
Conversely, if $M_P$ is $\kappa$-saturated and satisfies conditions
(1), (2), and (3), then $M_P$ satisfies the Density Property.
\end{lemma}
\begin{proof}
Assume that $M_P$ satisfies the Density property.
(1) and (2) follow from Remark~\ref{rem:density-substructure}.
For (3), let $V \subset M$ be $T$-definable with parameters~$\av$, such that $\U(V) = 1$.
Notice that $V$ does not fork over any set,
because $\U(M) = 1 = \U(V)$, 
and in particular $V$ does not fork over $P(\av)$.
Let $q(x) \in S^1_1(\av)$ expanding $x \in V$ and such that 
$q \notind_{P(\av)} \av$.
By the Density property, $q$~is realized in $P(M)$ and therefore $V \cap P(M)$
is nonempty. 

For the converse, assume that $M_P$ is $\kappa$-saturated and
satisfies the conditions in the lemma.
Let $A \subset M$ be very small and $q \in S^1_1(A)$ such that
$q \ind_{P(A)} A$.
If $\U(q) = 0$, then $\U\Pa{q \rest{P(A)}} = 0$, and hence, since
$P(M)$ is closed in~$M$, all realization of $q$ are in~$P(M)$, and in particular $q$ is realized in~$P(M)$.
Otherwise, $\U(q) = 1$.
Thus, for every $V \coloneqq \phi(x, \av) \in q(x)$, $\U(V) = 1$, 
and therefore $V \cap P(M) \neq \emptyset$.
Therefore, $q$~is finitely satisfiable in~$P(M)$, and hence, by saturation, 
$q$~is satisfiable in~$P(M)$.
\end{proof}

\begin{example}
There exists a pregeometric structure $\monster$ such that $\indacl\,$ is not
low.
In fact, let $\monster$ be a monster model of $T := \Theory(\pair{\Z, <})$.
We have $a \in \acl(b)$ iff $\abs{a - b}$ is finite, and
$\acl(B) = \bigcup_{b \in B} \acl(b)$.
We shall prove that, for every $M_P \models \Ttwo$, if $P(M) \neq M$, $P(M)$
is algebraically closed in~$M$, and 
$M_P$ is $\omega$-saturated, then $P$ is not dense in $M$ (and
therefore $M_P$ does not satisfy the Density Property).
Let $a \in M \setminus P(M)$, and let
$q(y)$ be the following partial $S^2_1$-type over $a$:
\[
a < y \et y - a = \infty \et [a,y] \cap P = \emptyset.
\]
Notice that $q(y)$ if finitely satisfiable in $M$, and hence, by saturation,
there exists $b \in M$ satisfying it.
Thus, $[a, b]$ is a $T$-definable infinite set that does not intersect~$P(M)$.
\end{example}

\begin{definition}
Let $f: \monster^m \leadsto \monster ^n$ be an \intro{application} (that is, a
multi-valued partial function); we say that $f$ is a \intro{\zapplication{}} if
$f$ is definable and $\U(f(\cv)) \leq 0$ is finite for every $\cv \in
\monster^m$. 
\end{definition}

\begin{definition}
We say that ``$\U$ is definable'' if, for every formula $\phi(\x, \y)$ and every $n \in \Nat$, the set
$\set{\bv \in \ \monster^m: \U \Pa{\phi(\monster^n, \bv)} = n}$ is definable, with the same parameters as~$\phi$.
\end{definition}

\begin{remark}
$\U$ is definable iff, for every formula $\phi(x, y)$ without parameters (where $x$ and $y$ have length~$1$), the set
$\set{b \in \ \monster: \U \Pa{\phi(\monster^n, b)} = 0}$ is definable without parameters.
\end{remark}

\begin{remark}
If $\monster$ is pregeometric, then $\Uacl$ is definable iff $\monster$
eliminates the quantifier~$\exists^\infty$.
\end{remark}

\begin{remark}
If $\U$ is definable, then loveliness is first-order.
The axioms of $\Td$ are:
\begin{description}
\item[Closure] $P(M)$ is closed in~$M$;
\item[Density] for every $V$ $T$-definable subset of~$M$,
if $\U(V) = 1$, then intersects~$P(M)$;
\item[Extension]
let $V$ be a $T$-definable subset of $M$ with $\U(V) = 1$ and 
$f: M^n \leadsto M$ be any $T$-definable \zapplication;
then, $V \nsubseteq f(P(M)^n)$.
\end{description}
In particular, if $\monster$ is geometric and $\ind = \indacl$, then the axioms of $\Td$ are:
\begin{description}
\item[Closure]
$P(M) \prec M$;
\item[Density]
for every $V$ $T$-definable subset of~$M$,
if $V$ is infinite, then $V$ intersects~$P(M)$;
\item[Extension]
let $V$ be a $T$-definable infinite subset of $M$ with $\U(V) = 1$ and $\bv$ be a finite tuple in~$M$;
then, $V \nsubseteq \mat(\bv P(M))$.
\end{description}
In both cases, if $\monster$ expands an integral domain, then the Extension axiom can be proved from the first two axioms (\cite[Theorem~8.3]{fornasiero-matroids}).
\end{remark}


Conversely, we have the following result.
\begin{proposition}
Assume that loveliness is first-order.
Then, $\U$~is definable.
\end{proposition}
\begin{proof}
Fix an $\Lang$-formula $\psi(x, y)$.
Denote $V_c \coloneqq \phi(\monster, c)$,
and define $Z \coloneqq \set{c \in \monster: \U(V_c) = 1}$.
We have to prove that $Z$ is definable.
This is equivalent to show that both $Z$ and its complement are preserved under ultraproducts.

\begin{enumerate}
\item
Let $(b_i : i \in I)$ be a sequence such that $b_i \in Z$ for every $i \in I$.
For every $i \in I$, choose $c_i \in \monster$ such that $c_i \in V_c \setminus \mat(\emptyset)$.
By Lemma~\ref{lem:lovely-existence}, for every $i \in I$ there exists a lovely pair $\pair{N_i, P_i}$ such that
$b_i \in P_i$ and $c_i \ind_{b_i} P_i$.
Therefore, $c_i \notin P_i$.
Let $\pair{N, P(N), b, c}$ be an ultraproduct of the $\pair{N_i, P_i, b_i, c_i}$ and let
$\pair{M, P(M)}$ be a $\kappa$-saturated elementary extension of $\pair{N,
P(N)}$. 
Thus, $c \notin P(M)$.
If, for contradiction, $\U(V_b) = 0$, then $c \in \mat(b)$.
However, $b \in P(M)$ and $P(M)$ is closed in~$M$, absurd.\\
Notice that for this half of the proof we only used the fact that the Density property is first-order.
Notice moreover that we proved that $\mat$ is a definable matroid in the sense of \cite{fornasiero-matroids}.
\item
Let $(b_i : i \in I)$ be a sequence such that $b_i \notin Z$ for every $i \in I$.
By Lemma~\ref{lem:lovely-existence}, for every $i \in I$ there exists a lovely pair $\pair{N_i, P(N_i)}$ such that
$b_i \in P(N_i)$.
Let $\pair{N, P(N), b}$ be the ultraproduct of the $\pair{N_i, P(N_i), b_i}$ with ultrafilter~$\mu$,
and let $\pair{M, P(M)}$ be a $\kappa$-saturated elementary extension of $\pair{N, P(N)}$.
By the first half of the proof, the set $\mat(b)$ is ord-definable.
Hence, ``$x \in V_b$ and $x \notin \mat(b)$'' is a consistent partial type (in~$x$, with parameter~$b$).
Thus, by the Extension property, there exists $d \in M$ such that $d \in V_b$, $d \notin \mat(b)$, and  $d \ind_b P(M)$.
Since $P(M)$ is closed in~$M$, we have $d \notin P(M)$.
Thus, there exists $c \in N$ such that $c \in V_b \setminus P(N)$.
Choose $(c_i: i \in I)$ such that $c  = (c_i: i \in I) / \mu$, and such that $c_i \in V_{b_i} \setminus P(N_i)$ for every $i \in I$.
However, since $\U(V_{b_i}) = 0$, we have $c_i \in \mat^{N_i}(b_i) \subseteq P(N_i)$, absurd.
\qedhere
\end{enumerate}
\end{proof}

See \cite{fornasiero-matroids} for more results on the case when $\U$ is definable, 
and \cite{boxall} for more on lovely pairs of geometric structures.


\section{NIP, stability, etc. in lovely pairs}
\label{sec:stability}

For this section, we assume that $\Mp = \pair{\monster, P}$ 
is a monster model of $\Ttwo$.

\subsection{Coheirs}


Let $\M \subseteq \N \subset \Mb$ be small subsets of~$\Mb$,
such that $\pair{\M, P(\M)} \preceq \pair{\N, P(\N)} \prec \Mb_P$. Assume also
that both $\pair{\M, P(\M)}$ and $\pair{\N, P(\N)}$ are sufficiently saturated
(in particular, they are $\kappa_0$-saturated).

\begin{remark}

$\M \ind_{P(\M)} P$, and similarly for~$\N$. 
\end{remark}
\begin{proof}
It is sufficient to prove that $\bar{m}\ind_{P(\M)}P$ for every finite tuple
$\bar{m}$ from~$\M$. By local character there is some $C\subseteq P$ with
$\card C < \kappa_0$ and such that $\bar{m}\ind_C P$. 
Let $C^\prime\models \tp^2(C/\bar{m})$ be such that $C^\prime\subseteq \M$. Then $C^\prime\subseteq P(\M)$ and $\bar{m}\ind_{C^\prime}P$. Therefore $\bar{m}\ind_{P(\M)}P$.
\end{proof}

\begin{lemma}\label{lem:62}
Let $a \in \Mb^h$ and $q$ be a small tuple in~$P$.
Assume that $a \ind_{\M P} \N$ and $a \M \ind_{P(\M) q} P$.
Then, $a \N \ind_{P(\N) q} P$.
\end{lemma}
\begin{proof}
\begin{align}
a \M \ind_{P(\M) q} P \Ra 
a \ind_{\M q} P  \Ra
& a \ind_{\M q} \M P.\\
a \ind_{\M q} \M P \ett a \ind_{\M P} \N P  \Ra 
& a \ind_{\M q} \N P.\\
a \M \ind_{\M q} \N P \Ra 
a \M \ind_{\N q} \N P \Ra
& a \N \ind_{\N q} P.\\
\pair{\N, P(\N)} \prec \pair{\Mb, P} \Ra
\N \ind_{P(\N)} P \Ra
& \N q \ind_{P(\N) q} P.\\
a \N \ind_{\N q} P \ett \N q \ind_{P(\N) q} P \Ra
& a \N \ind_{P(\N) q} P.
\end{align}
\end{proof}

Remember that $\Td$ is the theory of lovely pairs.
\begin{proviso*}
For the remainder of this section, we assume that \textbf{``being lovely'' is
  a first order property} and that $\Mp = \pair{\monster, P}$ is a monster
model of~$\Td$.
\end{proviso*}

\subsection{NIP}

\begin{theorem}
If $T$ has NIP, then $\Td$ also has NIP.
\end{theorem}

\begin{proof}[Short proof]
Apply Proposition~\ref{prop:model-complete} and \cite[Corollary~2.7]{CS10}.
\end{proof}
\begin{proof}[Long proof]
Let $a\in \monster$. Suppose $\tp^2(a/\N)$ is finitely satisfied in $\M$. 
We show that there are no more than $2^{|\M|}$ choices for $\tp^2(a/\N)$.  

By local character there is some $C\subset \N P$ such that 
$\card C \leq \kappa_0$ and $a\ind_C \N P$. 
Therefore there is some $\M^\prime$ such that $\pair{\M,P(\M)}\prec
\pair{\M^\prime,P(\M^\prime)}\prec \pair{\N,P(N)}$, $|\M^\prime|=|\M|$ and
$a\ind_{\M^\prime P}\N P$. 
This $\M^\prime$ depends on $a$. However we are assuming
$\tp^2(a/\N)$ is finitely realisable in $\M$ which implies that $\tp^2(a/\N)$ is
invariant over $\M$. Therefore any $\M^{\prime\prime}\models
\tp^2(\M^\prime/\M)$ such that $\M^{\prime\prime}\subseteq \N$ would work in
place of $\M^\prime$. There are no more than $2^{|\M|}$ possibilities for
$\tp^2(\M^\prime/\M)$. For each of these, fix some particular realisation.
Then, when we choose $\M^\prime$, we are actually selecting it from a list of
at most $2^{|\M|}$ things.

First we select the correct $\M^\prime$ from our list and assert that
$a\ind_{\M^\prime P}\N P$. We now choose
$\tp^2(a\bar{f}/\M^\prime)$ extending $\tp^2(a/\M^\prime)$ such that
$a\M^\prime\bar{f}$ is \Pindependent. 
We know we can do this such that 
$\card{\bar{f}}\leq \card{\M^\prime}= \card{\M}$. 
This was also a choice of one thing from a list of no more than
$2^{\card{\M}}$ things. 
Let $\bar{f}^\prime$ be such that $\bar f \in P$,
$a\bar{f}^\prime\models \tp^2(a\bar{f}/\M^\prime)$,
 and $\tp^1(a\bar{f}^\prime/\N)$ is finitely realisable in $\M^\prime$. 
We specify $\tp^1(a\bar{f}^\prime/\N)$ and we know that this too is a choice
from a list of no more than $2^{\card{\M}}$ things (since the original theory
$T$ has NIP). 
We now have enough to completely determine $\tp^2(a\bar{f}^\prime/\N)$. This is
because the choice of $\tp^2(a\bar{f}^\prime/\M^\prime)$ determines
$\Ptp(a\bar{f}^\prime)$ and, by Lemma~\ref{lem:62}, it also gives that
$a\N\bar{f}^\prime$ is \Pindependent and then we only need
$\tp^1(a\bar{f}^\prime/\N)$ to determine $\tp^2(a\bar{f}^\prime/\N)$. Overall,
we made our choice from a list of $2^{|\M|}\times 2^{|\M|}\times
2^{|\M|}=2^{|\M|}$ things.
\end{proof}

\subsection{Stability, super-stability,  \texorpdfstring{$\omega$}{omega}-stability}

\begin{theorem}
If $T$ is stable (resp. superstable, resp.\ totally transcendental), then
$\Td$ also is. 
\end{theorem}
\begin{proof}
Assume that $T$ is stable.
Remember that a theory $T$ is stable iff it is $\lambda$-stable for some
cardinal $\lambda$, iff it is $\lambda$-stable for every 
$\lambda^{\card T} = \lambda$.
Choose $\lambda$ a small cardinal such that 
$\lambda^{\kappa_0 + \card T} = \lambda$.
Let $\pair{M, P(M)} \prec \Mp$ be a model of~$\Td$ of cardinality~$\lambda$.
Notice that $\card{S^1_{\kappa_0}(M)} = \lambda$.
We must prove that $\card{S^2_1(M)} \leq \lambda$.
Let $q \in S^2_1(M)$ and $c \in \monster$ satisfying~$q$.
Let $\pv \subset P$ such that $c \ind_{M \pv} P$ and 
$\card{\pv} < \kappa_0$. 
Since moreover $M \ind_{P(M)} P$, we have $c \ind_{P(M) \pv} P$, and therefore
$c M \pv$ is \Pindependent.
Thus, $\tp^2(c M \pv)$ is determined only by $\tp^1(c M \pv)$ plus the
\Ptype of~$c$.
Therefore, $\tp^2(c /M)$ is determined by $\tp^1(c /M)$ plus the \Ptype
of~$c$.
Since $\card{\pv} < \kappa_0$, we have
$\card{S^2_1(M)} \leq \card{S^1_{\kappa_0}(M)} = \lambda$, and we are done.

Assume now that $T$ is super-stable.
Remember that a theory $T$ is super-stable iff there exists a cardinal $\mu$
such that $T$ it is $\lambda$-stable every
cardinal $\lambda > \mu$, and that $T$ is totally transcendental iff we can
take  $\mu = \card T$.
Moreover, since $T$ is super-stable, $\ind$ is superior, and therefore
$\kappa_0 = \omega$.
Let $\lambda \geq \mu$ and
$\pair{M, P(M)} \prec \Mp$ be a model of~$\Td$ of cardinality~$\lambda$.
Let $\pv \subset P$ such that $c  \ind_{M \pv} P$ and $\pv$ is finite.
As before, $c M \pv$ is \Pindependent, and thus
$\tp^2(c /M)$ is determined by $\tp^1(c /M)$ plus the \Ptype of~$c$.
Since $\pv$ is finite, we have
$\card{S^2_1(M)} \leq \card{S^1_{< \omega}(M)} = \lambda$.
Hence, $\Td$ is super-stable, and, if $T$ is totally transcendental, then
$\Td$ is also totally transcendental.
\end{proof}


\begin{example}
In general, if $T$ is geometric, $\Td$ is not geometric.
For instance, if $\monster$ is an o-minimal structure expanding a field and
$\ind = \ind[$\acl$]$, then $\Td$ is not geometric~\cite{fornasiero-matroids}.
\end{example}

\subsection{Simplicity and supersimplicity}
\label{subsec:simple}

By ``divides'' we will always mean ``divide in the sense of Shelah's'' (but we
might have to specify in which structure).

\begin{fact}
\label{fact:simple-definition}
The following are equivalent:
\begin{enumerate}
\item 
$\tp(\cv / A \bv)$ does not divide over $A$;
\item 
for any indiscernible sequence $I = \pair{\bv_i: i \in \Nat}$
with $\bv_0 = \bv$, let $p_i(x)$ be the copy of $\tp(\cv/A \bv)$ over 
$A \bv_i$; then, there is a tuple $\cv'$ realizing $\bigcup_i p_i(x)$;
\item
for any indiscernible sequence $I = \pair{\bv_i: i \in \Nat}$, with $\bv_0 =
\bv$, there is a tuple $\cv'$ realizing $\tp(\cv / A \bv)$, such that $I$ is
indiscernible over  $A \cv'$.
\end{enumerate}
\end{fact}
\begin{proof}
\cite[Remark~3.2(2) and Lemma~3.1]{casanovas07}.
\end{proof}

The following fact is well known: for a reference, see
\cite[Exercise~29.1]{TZ}.

\begin{fact}\label{fact:simple}
Let $T$ be a simple theory.
Let $\pair{M_i: i < \omega}$ be an indiscernible sequence over~$A$.
Assume that $C \indf_{A} M_0$.
Let $p_0(y) \coloneqq \tp(C/M_0)$ and $p_i(y)$ be the copy of $p_0$
over~$M_i$.
Then, there exists~$C'$, such that:
\begin{enumerate}
\item $C' \models \bigcup_i p_i(y)$;
\item $C' \indf_A \bigcup_i M_i$;
\item $\pair{M_i: i < \omega}$ is an indiscernible sequence over~$A C'$.
\end{enumerate}
\end{fact}

Remember that $A \indf_B C$ implies $A \ind_B C$.
\begin{remark}\label{rem:heir}
Assume that $\pair{M, P(M)} \preceq \pair{\monster, P}$ (but not necessarily
that $\pair{\monster, P} \models \Td$).
Then, for every $\cv \in P$, $\tp^2(\cv/M)$ is finitely satisfiable in~$P(M)$.
Therefore, $M \indf_{P(M)} P$ (in the sense of both $T$ and~$\Td$),
and $M \ind_{P(M)} P$.
\end{remark}


\begin{proposition}\label{prop:simple}
If $T$ is simple, then $\Td$ is also simple.
If $T$ is supersimple, then $\Td$ is also supersimple.
\end{proposition}

\begin{proof}
The proof is almost identical to the one of~\cite[Proposition~6.2]{BPV}.
We will use the notation $A \indf_B C$ to mean that $A$ and $C$ do not fork
over~$C$, in the sense of Shelah's, according to the theory~$T$
(and \emph{not} to the theory~$\Td$), while, when saying ``$\tp^2(\av/B)$
divides  over~$C$'', we will imply ``according to~$\Td$''.

Let $\pair{M, P(M)}$ be a small model of~$\Td$ and $\av$ be a finite tuple.
We have to find $A \subseteq M$ such that $\card A \leq \card T$,
and $\tp^2(\av/M)$ does not divide over~$A$ (by \cite[Lemma~6.1]{BPV}, 
this will prove that 
$\Td$ is simple).
When $T$ is supersimple, we will see that $A$ could be chosen finite (and
hence $\Td$ is supersimple).

By Remark~\ref{rem:heir}, $M \indf_{P(M)} P$.
Hence, since $\indf$ satisfies local character, there exist
$C \subset P$ and $A \subset M$, both of cardinality at most~$\card T$,
such that
\begin{equation}\label{eq:simple-1}
\av \indf_{A C} M P.
\end{equation}
Moreover, since $C \subset P$,
$C \indf_{P(M)} M$, and hence, after maybe enlarging $A$,
we can also assume that
\begin{equation}\label{eq:simple-2}
C \indf_{P(A)} M.
\end{equation}
If moreover $T$ was supersimple, then $A$ and $C$ could be chosen finite.

Let $\pair{M_i: i < \omega}$ be an $\Ltwo$-indiscernible sequence over~$A$,
such that $M_0 = M$.
\eqref{eq:simple-2} implies that $C \indf_A M$;
therefore, we can apply Fact~\ref{fact:simple}.
Let $p_0(y) \coloneqq \tp^1(C/M_0)$ and $p_i(y)$ be the copy of $p_0$
over~$M_i$.
Then, there exists~$C'$ realizing the conclusions of Fact~\ref{fact:simple}.
Notice that $C' \indf_{P(A)} A$, thus, by transitivity,
$C' \indf_{P(A)} \bigcup_i M_i$, and hence
$C' \ind_{P(A)} \bigcup_i M_i$.
Since $\Mp$ is lovely, there exists $C''$ in $P$ such that
$C'' \elem^1_{\bigcup_i M_i} C'$; thus, \wloG $C' \subset P$.

Notice that $M \ind_{P(M)C} P$, thus $M C$ is \Pindependent,
and the same for $M C'$.
Moreover, $M C$ and $M C'$ satisfy the same $\Lang$-type and the same
\Ptype: therefore, they have the same $\Ltwo$-type.
Thus, by changing the sequence of $M_i$'s, we can assume that $C = C'$.
Let $r(\x) \coloneqq \tp^1(\av / M C)$,
and $r_i(\x)$ be the copy of $r(x)$ over $M_i$.
By \eqref{eq:simple-1}, $\av \indf_{A C} M C$, 
moreover, $\pair{M_i: i < \omega}$ is $\Lang$-indiscernible over~$A C$.
Thus, there exists $\av'$ realizing $\bigcup_i r_i(\x)$,
such that
\begin{equation}\label{eq:simple-3}
\av' \indf_{A C} \bigcup_i M_i C.
\end{equation}
By loveliness of $\Mp$ again, we can assume that
$\av' \ind_{\bigcup_i M_i C} P$ (note: here we have~$\ind$, not~$\indf$).
Thus, for each~$i$, by~\eqref{eq:simple-3},
\begin{equation}
\av' \ind_{M_i C} P.
\end{equation}
Since moreover $C \subset P$ and $M_i  \ind_{P(M_i)} P$,
we have $M_i C \ind_{P(M_i) C} P$, and therefore
$\av' M_i C \ind_{P(M_i) C} P$.
Thus, each $\av' M_i C$ is \Pindependent.
Moreover, they have all the same \Ptype and the same $\Lang$-type.
Thus, all the $\av' M_i C$ have the same $\Ltwo$-type. 
Moreover, by~\eqref{eq:simple-1}, $\av \ind_{M C} P$, and thus
$\av M C$ is \Pindependent, and therefore 
$\av M C \equiv^2 \av' M_i C$.
Therefore, by Fact~\ref{fact:simple-definition},
$\tp^2(\av/M)$ does not
divide over~$C$, and we are done.
\end{proof}

\begin{question}
Assume that $T$ is (super)rosy.
Is $\Td$ also (super)rosy?
\end{question}


\section{Independence relation in lovely pairs}
\label{sec:lovely-independence}

In this section, we will assume that ``being lovely is
first order'' and that $\Mp = \pair{\monster, P}$ is a monster model of $\Td$.

Let $\ind'$ the following relation on subsets of $\Mp$:
$A \ind'_C B$ iff $A \ind_{C P} B$;
we will write $\ind_P$ instead of $\ind'$.

\begin{definition}
$\ind$ satisfies (*) if:\\ 
For every $a \in \Mb^k$, $b \in \Mb^h$,
for every $C$ tuple in $\Mb$ (not necessarily of small length), 
if $a \notind_C b$, then there exists a finite subtuple $c$ of $C$ (of
length~$l$) and an $\Lang$-formula $\phi(x, y, z)$,
such that:
\begin{enumerate}
\item $\Mb \models \phi(a, b, c)$;
\item for every $a' \in \Mb^k$,  for every $c'$ subtuple of $C$ of length~$l$,
if $\Mb \models \phi(a', b, c')$, then $a' \notind_{C} b$.
\end{enumerate}
\end{definition}
Notice that (*) implies Strong Finite Character.

\begin{remark}
the independence relations in
examples~\ref{ex:independence-strict} and~\ref{ex:independence-matroid} 
satisfy (*).
\end{remark}
\begin{proof}
Let us show that when $T$ is simple, then $\indf$ satisfies (*)
Assume that $a \notind_C b$.
Let $p(y) \coloneqq \tp(b /C a)$ and $p_0 \coloneqq \tp (b /C)$.
Thus, $p$~divides over~$C$.
Hence, by \cite[Proposition~2.3.9 and Remark~2.3.5]{wagner}, there exist
a formula~$\psi$, a cardinal~$\lambda$, a finite subtuple $c \subset C$, 
and a formula $\theta(y, a, c) \in p(y)$, such that
such that $D_0(\theta(y, a, c) < D_0(p_0)$, 
where $D_0(\cdot) \coloneqq D(\cdot, \psi,\lambda)$.
Let $n \coloneqq D_0(p_0)$.
By \cite[Remark~2.3.5]{wagner}, the set 
$\set{a', c': D_0(\theta(y, a', c')) <   n}$ is ord-definable; 
hence, there exists a formula $\sigma(x, z)$ such that
$\monster \models \phi(a, c)$, and for every
$c'$ and~$a'$, if $\phi(a', c')$, then $D(\theta(y, a', c')) < n$.
Define $\phi(x, y, z) \coloneqq \theta(y, x, z) \et \sigma(y, z)$.
Then, $\phi$ satisfies the conclusion of~(*).

The other cases are similar: when $T$ is rosy, use the local \th-ranks
instead of the rank $D$ (\cite[\S3]{onshuus06}); when $\mat$ is an
existential matroid, to prove (*) for $\indmat$ use the associated global rank.
\end{proof}

\begin{proposition}
$\ind_P$ is an independence relation on $\Mb_P$; the constant $\kappa_0$ for
the local character axiom is the same for $\ind_P$ and for~$\ind$.
Besides, the closure operator $\mat_P$ induced by $\ind_P$ satisfies
$\mat_P(X) = \mat(X P)$; in particular, $\mat_P(\emptyset) = P$.
Moreover, if $\ind$ satisfies \rom(*\rom), then $\ind_P$ also satisfies \rom(*\rom).
\end{proposition}
In particular, one can consider $\ind_P$-lovely pairs.
\begin{proof}
Let us verify the various axioms of an independence relation.
Invariance is clear from invariance of~$\ind$.
Symmetry, Monotonicity, Base Monotonicity, Transitivity, Normality,
Finite Character and Local Character (with the same constant $\kappa_0$)
follow from Remark~\ref{rem:localization}.

It remains to prove Extension; instead, we will prove the Existence Axiom
(which, under the other axioms, is equivalent to Extension 
\cite[Exercise 1.5]{adler}).
Let $A$, $B$ and $C$ be small subsets of $\Mb$.
Let $P_0 \subset P$ be a small subset, such that $P(A B C) \subseteq P_0$ and
$A B C \ind_{P_0} P$ (here we use that Local Character holds also for large
subsets of~$\monster$).
\Wlog, we can assume that $P_0 \subset A \cap B \cap C$.
Let $A' \equiv^1_{C} A$ such that $A' \ind_{C} B$.
Since $\Mb(P)$ is a lovely pair,
there exists $A'' \equiv^1_{C B} A'$ such that $A'' \ind_{C B} P$.
Notice that $A'' \ind_{C} B$.
Hence, by some forking calculus, $A'' \ind_{C P} B$.
Moreover, $A'' C \equiv^1 A C$,
and $A C$ is \Pindependent.
We claim that $A'' C$ is also \Pindependent:
in fact, $C B \ind_{P_0} P$ and $A'' \ind_{B C} P$, and hence
$A'' B C \ind_{P_0} P$.
Since $P$ is closed, this also implies that $P(A'' B C) = P_0$, and
thus, since $\Mb(P)$ is lovely, $A'' C \equiv^2 A C$, and we are
done. 

Assume now that $\ind$ satisfies (*).
Let $a$ and $b$ be finite tuples in~$\Mb$,
and $C$ be a (not necessarily small) subset of~$\Mb$.
Assume that $a \notind_{P C} B$.
Then, by (*), there exists an $\Lang$-formula $\phi(x, y, z, w)$ and finite
tuples $c$ in $C$ and $p$ in $P$, such that
$\Mb \models \phi(a, b, c, p)$ and, for every $a' \subset \Mb$, $c' \subset C$ 
and $p' \subset P$, if $\Mb \models \phi(a', b, c', p')$,
then $a' \notind_{C  P} b$. 
Let $\psi(x, y, z)$ be the $\Ltwo$-formula
$(\exists w \in P) \phi(x, y, z)$.
Then, $\psi$ witnesses the fact that $\ind_P$ satisfies (*).
\end{proof}

\begin{conjecture}
There exists an independence relation $\ind[2]$ on~$\Mp$, such that:
\begin{enumerate}
\item $\ind[2]$ coincides with $\ind$ on subsets of $P$;
\item $\ind[2]_P = \ind_P$.
\end{enumerate}
\end{conjecture}

\subsection{Rank}

Assume that $\ind$ is superior (that is, $\kappa_0 = \omega$).
We have seen that then $\ind_P$ is also superior.
Let $\U \coloneqq \Uind$ be the rank on $\monster$ induced by $\ind$ and $\Up$ the rank on~$\Mp$
induced by $\ind_P$.
We now investigate the relationship between the 2 ranks.
Given a partial $\Lang$-type $\pi(\x)$, we define $U(\pi)$ as the supremum of
the $\U(q)$, where $q(\x)$ varies among the complete $\Lang$-types extending
$\pi(\x)$, and similarly for $\Up$ on partial $\Ltwo$-types.
However, if $q$ is a complete $\Lang$-type, then $q$ is also a partial
$\Ltwo$-type.
Hence, we can compare $\U(q)$ and $\Up(q)$.

\begin{lemma}[{\cite[8.31]{fornasiero-matroids}}]
For every $B \subset \monster$ small and every $q \in S^1_n(B)$,
$\U(q) = \Up(q)$.
The same equality holds for partial $\Lang$-types.
\end{lemma}
\begin{proof}
First, we will prove, by induction on $\alpha$, that, for every ordinal number
$\alpha$, and every $\Lang$-type~$q'$, if
$\Up(q') \geq \alpha$, then $\U(q') \geq \alpha$.
If $\alpha$ is limit, the conclusion follows immediately from the inductive
hypothesis.
Assume that $\alpha = \beta + 1$, and that $q' = q$.
Let $C \supset B$ be a small set and $\av \in \monster^n$, such that
$\Up(\av / C) \geq \beta$ and $\av \notind_{PB} C$ ($C$ and $\av$ exist by
definition of~$\Up$).
By inductive hypothesis, $\U(\av / C) \geq \beta$.
Let $P_0 \subset P$ be small (actually, finite), such that 
$\av \ind_{C P_0} P$; define $C' := C P_0$.
Notice that $\av \notind_{P B} C'$ and $\av \ind_{C'} P$.
\begin{claim}
$\av \notind_B C'$.
\end{claim}
If not, then, since $B \subset C'$ and by transitivity, we would have $a
\ind_B C' P$, and therefore $a \ind_{P B} C'$, contradiction.

Moreover, since $\matP(C') = \matP(C)$, we have $\Up(a / C') = \Up(a / C) \geq
\beta$.
Hence, by Inductive Hypothesis, $\U(a / C') \geq \beta$.
Since $a \notind_B C'$, we have $\U(a  / B) \geq \beta + 1= \alpha$, and we
are done.
Therefore, $\U(q) \geq \Up(q)$.

Second, we will prove by induction on~$\alpha$, that, for every ordinal 
number~$\alpha$, and every $\Lang$-type~$q'$, if
$\U(q') \geq \alpha$, then $\Up(q') \geq \alpha$.
If $\alpha$ is a limit ordinal, the conclusion is immediate from the inductive
hypothesis.
If $\alpha = \beta + 1$, let $C' \supset B$ be a small set
and $r' \in S^1_n(C')$, such that $q \forkext r'$ and $\U(r') \geq \beta$.
By the Extension Property, there exists $C \elem^1_B C'$ such that 
$C \ind_B P$; let $f \in \aut(\monster/B)$ such that $f(C') = C$, and let $r
:= f(r')$.
Then, $q \forkext r$ and $\U(r) \geq \beta$.
By inductive hypothesis, $\Up(r) \geq \beta$.
Let $a \in \monster$ be any realization of~$r$.
If $a \ind_{PB} C$, then, since $C \ind_B P$, we have $a \ind_B C$,
contradicting \mbox{$q \notind_B C$}.
Thus, $a \notind_{PB} C$, and hence
then $\Up(q) \geq \Up(a/B) > \Up(r) \geq \beta$, and we are done.
Therefore, $\Up(q) \geq \U(q)$.

For the case when $q$ is a partial $\Lang$-type, let $r \in S^1_n(B)$ be any
complete type extending~$q$. Then, $\U(r) = \Up(r)$.
Since this is true for any such~$r$, we have $\U(q) = \Up(q)$.
\end{proof}


\subsection{Approximating definable sets}

We say that a ran $\U$ is \intro{continuous} if, for every ordinal~$\alpha$ and small set~$B$, the set $\set{q \in S_n(B) : \U(q) > \alpha}$ is closed in $S_n(B)$.

\begin{proposition}[{\cite[8.36]{fornasiero-matroids}}]
Assume that $\ind$ is superior and $\Up$ is continuous.
Let $\bv$ be a small \Pindependent tuple in~$\monster$.
Let $X \subseteq \monster^n$ be $T$-definable over~$\bv$.
Let $Y \subseteq X$ be $\Td$-definable over~$\bv$.
Then, there exists $Z \subseteq X$ $T$-definable over~$\bv$,
such that, for every $\cv \in Z \Sdiff Y$, $\Up(\cv / \bv) < \Uind(X)$.
\end{proposition}

\begin{proof}
Let $\alpha \coloneqq \Uind(X)$.
Let $W \coloneqq \set{q \in S^2_X(\bv) : \Up(q) = \alpha}$.
Since $\Up$ is continuous, $W$~is closed.
Let $\theta: S^2_X(\bv) \to S^1_X(\bv)$ be the restriction map and $V \coloneqq \rho(W)$.
Define $\tilde Y \coloneqq \rho\Pa{S^2_Y(\bv) \cap W} \subseteq V$.
By Transitivity and Proposition~\ref{prop:back-and-forth}, $\rho$~is a homeomorphism between $W$ and~$V$, and therefore $\tilde Y$ is clopen in~$W$, and $W$ is closed in $S^1_{X}(\bv)$.
By standard arguments, there exists $Z \subseteq \monster^n$ which is $T$-definable over $\bv$ and such that $S^1_Z(\bv) \cap V = \tilde Y$.
Then, $Z$ satisfies the conclusion.
\end{proof}

\section{Producing more independence relations}
\label{sec:new-independence}

For this section, we assume that loveliness is first-order and that $\Mp =
\pair{\monster, P}$ is a monster model of~$\Td$.
Moreover, we assume that $\ind$ is superior, and we denote 
$\U \coloneqq \Uind$.

\begin{remark}
Let $B$ be a small subset of $\monster$ and $q$ be a complete type over~$B$.
Assume that $\alpha \leq \U(q)$.
Then, there exists a complete type $r$ extending~$q$, such that 
$\U(r) = \alpha$.
\end{remark}
\begin{proof}
By induction on $\beta := \U(q)$.
If $\beta = \alpha$, let $r := q$.
Otherwise, $\beta > \alpha$.
If, for every $r$ forking extension of~$q$, 
$\U(r) < \alpha$, then , by definition,
$\U(q) \leq \alpha$, absurd.
Hence, there exists $r$ forking extension of~$q$, such that $\U(r) \geq
\alpha$.
Thus, $\alpha \leq \U(r) < \beta$, and therefore, by inductive hypothesis,
there exists $s$ extension of~$r$, such that $\U(s) = \alpha$.
\end{proof}

Let $\theta$ be an ordinal such that $\theta = \omega^\delta$ for
some~$\delta$. 
We will use the ``big $\bigo$'' and ``small $\smallo$'' 
notations: $\alpha = \smallo(\theta)$ (or $\beta \gg \alpha$)
if $\alpha < \theta$,
$\alpha = \beta + \smallo(\theta)$ if there exists $\varepsilon <
\theta$ such that $\alpha = \beta + \varepsilon$, 
and $\alpha = \bigo(\theta)$ if there exists $n \in \Nat$ such that
$\alpha \leq n \theta$.
Notice that, since $\theta$ is a power of $\omega$, 
$\smallo(\theta) + \smallo(\theta) = \smallo(\theta)$.

Define $\indtheta$, the coarsening of $\ind$ at~$\theta$, 
in the following way:   
for every $\av$ finite tuple in $\monster$ and every $B$, $C$ 
small subsets of $\monster$,
$\av \indtheta_B C$ if $\U(\av / B) = U(\av/B C) + \smallo(\theta)$.
If $A$ is a small subset of $\monster$, define $A \ind_B C$ if, for every
$\av$ finite tuple in~$A$, $\av \indtheta_B C$.
We will use also the notation $(\ind)^\theta$ for~$\indtheta$.

Notice that $\indtheta$ is trivial if $\theta$ is large enough.
Assume that
\begin{itemize}
\item[(**)] For every $a \in \monster$, $\U(a / \emptyset) = \bigo(\theta)$.
\end{itemize}

Notice that Condition (**) is equivalent to:
\begin{itemize}\item[] 
For every $\av$ finite tuple in $\monster$ and
every $B \subset \monster$, $\U(\av / B) = \bigo(\theta)$.
\end{itemize}
However, it might happen that $\U(\monster) \gg \theta$; for instance, if
$\theta = 1$, it can happen that $\U(a/\emptyset)$ is finite for every $a
\in \monster$ and, for every $n \in \Nat$ there exists $a_n \in \monster$ such
that $\U(a / \emptyset)> n$, and thus $\U(\monster) = \omega \gg 1$.

\begin{proposition}
$\indtheta$ is a superior independence relation on~$\monster$, and $\ind$
refines~$\indtheta$. 
\end{proposition}
\begin{proof}
The only non-trivial axiom is Left Transitivity, which is also the only place
where we use the Condition (**) 
(Right Transitivity instead is always true).
Assume that $\cv \indtheta_B A$ and $\dv \indtheta_{B \cv} A$, for some finite
tuples $\cv$ and $\dv$, and some small sets $A$ and~$B$.
We claim that $\cv \dv \indtheta_{B} A$.
In fact, by Lascar's inequalities, we have
\begin{multline*}
\U(\cv \dv / B) \leq \U(\dv/ \cv B ) \oplus \U(\cv / B) \leq\\
\U(\dv / \cv \bv A) + o(\theta) \oplus \U(\cv / B A) + \smallo(\theta) =
\U(\dv / \cv \bv A) \oplus \U(\cv / B A) + \smallo(\theta).
\end{multline*}
By (**), we have that
$\U(\dv / \cv \bv A) \oplus \U(\cv / B A) = \U(\dv / \cv \bv A) + \U(\cv / B A)
+ \smallo(\theta)$.
Hence, again by Lascar's inequalities,
$\U(\cv \dv / B) \leq \U(\dv \cv / BA) + \smallo(\theta)$.

Symmetry then follows (see~\cite[Theorem~1.14]{adler}).
\end{proof}

Let $\mattheta$ be the closure operator induced by $\indtheta$: thus,
$a \in \mattheta(B)$ iff $\U(a/B) < \theta$.
Let $\Utheta$ be the rank induced by~$\indtheta$.

\begin{remark}
Let $B$ be a small subset of~$\monster$.
For every $q \in S_m(B)$, we have $\U(q) = n \theta + \smallo(\theta)$, for
a unique $n \in \Nat$.
Then, $\Utheta(q) = n$.
\end{remark}
\begin{proof}
First, we will prove, by induction on~$n$, that, if $\U(q) \geq n \theta$, then
$\Utheta(q) \geq n$.
If $n = 0$, there is nothing to prove.
Assume that we have already proved the conclusion for~$n$, and let $q \in
S_m(B)$ such that $\U(q) \geq (n + 1)\theta$.
Let $C \supset B$ small and $r \in S_M(C)$ extending $q$, such that
$\U(r) = n \theta$.
Thus, by inductive hypothesis, $\Utheta(r) \geq n$, and,
by definition, $r$ is a non-forking extension of $q$ in the sense
of~$\indtheta$.
Therefore, $\Utheta(q) \geq \Utheta(r) + 1  \geq n + 1$.

Conversely, we will now prove, by induction on~$n$, 
that, if $\Utheta(q) \geq n$, then $\U(q) \geq n \theta$.
Again, if $n = 0$, there is nothing to prove.
Assume that we have already proved the conclusion for~$n$, and let $q \in
S_m(B)$ such that $\Utheta(q) \geq n + 1$.
Let $r$ be a $\indtheta$-forking extension of $q$ such that $\U(r)
\geq n$. By inductive hypothesis, $\U(r) \geq n \theta$.
Moreover, by our choice of~$r$, $\U(q) \geq \U(r) + \theta \geq (n + 1)
\theta$. 
\end{proof}

Notice that there is at most one $\theta$ satisfying (**) and such that $\indtheta$ is nontrivial, that is the minimum ordinal satisfying (**).

\begin{example}
Let $T = ACF_0$.
Let $\Td$ be the theory of Beautiful Pairs for~$T$; that is 
$\pair{M, P(M)} \models \Td$ if $M \models ACF_0$ and $P(M)$ is a proper 
algebraically closed subfield of~$M$.
Then, $\Td$ is $\omega$-stable of $\U$-rank~$\omega$.
Let $\Mp$ be a monster model of~$\Td$.
We have two natural equivalence relations on $\Mp$: Shelah's forking $\indf$
and $\indp$, where $\ind$ is Shelah's forking relation for models of~$T$.
We claim that $(\indf)^\omega = \indp$.

In fact, according to both $(\indf)^\omega$ and $\indp$, $\pair{M, P}$ has
rank~1.
However, since $\pair{M, P(M)}$ expands a field, there is at most one
independence relation on it that has rank 1 \cite{fornasiero-matroids}.
\end{example}

\section*{Acknowledgements}
I am grateful to Hans Adler, Janak Ramakrishnan, Katrin Tent, Evgueni
Vassiliev, and Martin Ziegler for their help.
Special thanks to Gareth Boxall for the many discussions on the
subject: many of the ideas and of the proofs are by him.

\bibliographystyle{amsalpha}	
\bibliography{geometry}		

\end{document}